\documentclass[9pt,reqno]{amsart}

\usepackage[utf8]{inputenc}
\usepackage[english]{babel}
\usepackage{amsmath,amsfonts,amssymb,amsthm}
\usepackage[T1]{fontenc}
\usepackage{lmodern}
\usepackage{mathtools}

\usepackage[top=3.5cm,bottom=3.5cm,left=3.6cm,right=3.6cm]{geometry}

\usepackage[hyperindex=true,frenchlinks=true,
citecolor=Mahogany,linkcolor=DarkOrchid,urlcolor=Tan,linktocpage,
pagebackref=true]{hyperref}

\usepackage{tikz}
\usetikzlibrary{shapes}
\usetikzlibrary{fit}
\usetikzlibrary{decorations.pathmorphing}

\usepackage{wasysym}
\usepackage{cite}

\usepackage{enumitem}
\usepackage{multicol}

\title{Invariance principle via orthomartingale approximation}
\keywords{Random fields, invariance principle, orthomartingales, projective 
conditions, Maxwell and Woodroofe condition.}
\subjclass[2010]{60F05; 60F17; 60G10; 60G48; 60G60.}
\date{\today}
\author{Davide Giraudo}
\address{Normandie Universit\'e, Universit\'e de Rouen, 
Laboratoire de Math\'ematiques Rapha\"el Salem,
CNRS, UMR 6085, Avenue de l'universit\'e, BP 12, 
76801 Saint-Etienne du Rouvray Cedex, 
France.}
\email{davide.giraudo1@univ-rouen.fr}

\numberwithin{equation}{subsection}
\renewcommand{\leq}{\leqslant}
\renewcommand{\geq}{\geqslant}

\newtheorem{Theorem}{Theorem}[section]
\newtheorem{Th\'eor\`eme}{Th\'eor\`eme}[section]
\newtheorem{Proposition}[Theorem]{Proposition}
\newtheorem{Lemma}[Theorem]{Lemma}

\newtheorem{Definition}[Theorem]{Definition}
\newtheorem{D\'efinition}[Th\'eor\`eme]{D\'efinition}
\newtheorem{Corollary}[Theorem]{Corollary}

\theoremstyle{remark}
\newtheorem{Remark}[Theorem]{Remark}

\tikzstyle{Vertex}=[circle,draw=LimeGreen!80,fill=LimeGreen!8,
inner sep=1pt,minimum size=2mm,line width=1pt,font=\scriptsize]
\tikzstyle{Node}=[Vertex,draw=RoyalBlue!80,fill=RoyalBlue!8,inner sep=1.5pt]
\tikzstyle{Leaf}=[rectangle,draw=Black!70,fill=Black!16,
inner sep=0pt,minimum size=1mm,line width=1.25pt]
\tikzstyle{Edge}=[Maroon!80,cap=round,line width=1pt]
\tikzstyle{Mark1}=[draw=BrickRed!80,fill=BrickRed!8]
\tikzstyle{Mark2}=[draw=BurntOrange!80,fill=BurntOrange!8]
\tikzstyle{EdgeRew}=[->,RedOrange!80,cap=round,thick]

\newcommand{\h}{\mathcal H}
\newcommand \ens[1]{\left\{ #1\right\}}
\newcommand \R{\mathbb R}

\newcommand \N{\mathbb N}
\newcommand{\gr}[1]{\mathbf{#1}}
\newcommand{\imd}{\preccurlyeq}
\newcommand{\smd}{\succcurlyeq}

\newcommand{\E}[1]{\mathbb E\left[#1\right]}

\newcommand{\f}{\mathcal F}
\newcommand \Z{\mathbb Z}

\newcommand \abs[1]{\left|#1\right|}
\newcommand \eps{\varepsilon}

\newcommand{\Id}{\operatorname{I}}

\newcommand{\norm}[1]{\left\lVert #1 \right\rVert}
\newcommand{\pr}[1]{\left( #1\right)}

\newcommand{\til}[1]{\widetilde{#1}}

\begin{document}

\begin{abstract}
 We obtain a necessary and sufficient condition for the 
 orthomartingale-coboundary decomposition.
 We establish a sufficient condition for the approximation of 
 the partial sums of a strictly stationary random fields by those 
 of stationary orthomartingale differences. This condition can 
 be checked under multidimensional analogues of the 
 Hannan condition and the Maxwell-Woodroofe condition.
\end{abstract}

\maketitle 

\section{Introduction and notations}

In all the paper, we shall use the following notations. 
Let $\pr{\Omega,\f,\mu}$ be a probability space.
\begin{itemize}
 \item For a function $f\colon\Omega\to\R$, $\norm{f}$ will 
 denote the $\mathbb L^2$-norm of $f$. The subspace of centered 
 square integrable functions is denoted as $\mathbb L^2_0$.
 \item If $d$ is a positive integer, we denote by 
 $[d]$ the set $\ens{1,\dots,d}$.
 \item If $\gr{n}=\pr{n_1,\dots,,n_d}$ is an element of $\N^d$, we denote 
 by $\min\gr{n}$ the quantiy $\min_{1\leq q\leq d}n_q$ and $
 \abs{\gr{n}}:=\prod_{q=1}^dn_q$. Moreover, we shall write 
 $\gr{2^n}=\sum_{q=1}^d2^{n_q}\gr{e_q}$.
 \item If $q\in [d]$, then $\gr{e_q}$ is the element of $\N^d$ 
 such that the $q$th coordinate is equal to $1$, and all the others to $0$.
 \item We denote for an 
element $\gr{i}$ of $\Z^d$ and a non-empty subset $J$ of $[d]$ the multiindex 
$\gr{i_J}\in \Z^d$ defined as $\sum_{q\in J}i_q\gr{e_q}$.
\item Let $\pr{a_{\mathbf n}}_{\mathbf n\in \Z^d}$ be a family of real numbers. We 
 define 
 \begin{equation}\label{eq:definition_of_limsup}
  \limsup_{\mathbf n\to +\infty}a_{\mathbf n}:=
  \lim_{i\to +\infty}\sup_{\mathbf n:\min \mathbf n\geqslant i}a_{\mathbf n}.
   \end{equation}
\item We denote by $\imd$ the coordinatewise order, that is, for any 
$\gr{i}=\pr{i_q}_{q=1 }^d \in\Z^d$ and $\gr{j}=\pr{j_q}_{q=1 }^d \in\Z^d$, 
$\gr{i}\imd\gr{j}$ if and only if $i_q\leqslant j_q$ for any 
$q\in [d ]$.

\item Let $T_q$, $q\in [d]$ be bijective, bi-measurable and  measure preserving maps from $\Omega$ 
to itself which are pairwise commuting. For  $\gr{i}\in \Z^d$, we denote by $T^{\gr{i}}$ 
the map $T_1^{i_1}\circ\dots\circ T_d^{i_d}$, $U^{\gr{i}}\colon \mathbb L^1\to
\mathbb L^1$ the operator defined by $\pr{U^{\gr{i}}f}\pr{\omega}=
f\pr{T^{\gr{i}}\omega}$ and 
\begin{equation}\label{eq:definition_des_sommes_partielles}
 S_{\gr{n}}\pr{f}=\sum_{\gr{0}
\imd\gr{i}\imd\gr{n}-\gr{1}}U^{\gr{i}}\pr{f}=\sum_{\gr{0}
\imd\gr{i}\imd\gr{n}-\gr{1}} f\circ T^{\gr{i}}.
\end{equation}

 We also use the notation
$U_q:=U^{\gr{e_q}}$.
\item We shall write as a product the composition of operators $U_q$ 
and we use the convention $\prod_{q\in \emptyset}U_q=\Id$.
 \item If $I$ is a subset of $[d]$, then $\eps\pr{I}$ is the element of 
 $\Z^d$ whose $q$th coordinate is $-1$ if $q$ belong to $I$ and $1$ 
 otherwise. 
 \item The product, sum and minimum of two elements of $\Z^d$ is 
 understood to be coordinatewise.
 \item Let $\pr{\f_{\gr{i}}}_{\gr{i}\in\Z^d}$ denote a filtration. For 
 $J\subset [d]$, we denote by $\f_{\infty\gr{1_J}}$ the $\sigma$-algebra generated 
 by $\bigcup_{\gr{j}\in \Z^d,\gr{j_{ [d]\setminus  J}}\imd \gr{0}    }\f_{\gr{j}}$.
\end{itemize}

  \subsection{The invariance principle}
  
For $\mathbf i\succcurlyeq \mathbf 1$, we denote the unit cube with upper corner at 
$\mathbf i=\pr{i_1,\dots,i_d}$ that is,
\begin{equation}
 R_{\mathbf i}:=\prod_{q=1}^d\left(i_q-1,i_q\right].
\end{equation}

For a measurable function $f\colon \Omega\to \R$, we consider the partial 
sum process defined by 
\begin{equation}\label{eq:premiere_definition_processus_sommes_partielles_dim_d}
 S_{\mathbf n}(f,\mathbf t):=\sum_{\mathbf i\in \left[\mathbf 1,\mathbf n\right]}
 \lambda\left([\mathbf 0,\mathbf n\cdot \mathbf t]\cap R_{\mathbf i}\right)U^{\mathbf i}f, 
 \quad \mathbf t\in [0,1]^d,\mathbf n\in (\N^*)^d,
\end{equation}
where $\lambda$ denotes the Lebesgue measure on $\R^d$,
$[\mathbf 0,\mathbf n\cdot \mathbf t]=\prod_{q=1}^d[0,n_qt_q]$ and 
\begin{equation}
 \left[\mathbf 1,\mathbf n\right]=\ens{\mathbf i\in \Z^d,
1\leqslant i_q\leqslant n_q\mbox{ for each }q\in \ens{1,\dots,d}}.
\end{equation}
We are interested in the functional central limit theorem in 
$C\pr{[0,1]^d}$ for the net $\left( S_{\mathbf n}\pr{f,\cdot}\right)_{\mathbf n\in (\N^*)^d}$
in order to understand the asymptotic behavior of the partial 
sums of $(f\circ T^{\mathbf i})$ over rectangles. By "functional central 
limit theorem in $C\pr{[0,1]^d}$", we mean that for each continuous 
bounded functional $F\colon \pr{C\pr{[0,1]^d},\norm{\cdot}_\infty}
\to \R$, the convergence $F\left(S_{\mathbf n}(f,\cdot)/
a_{\mathbf n}\right)\to F\left(W\right)$ holds as $\min 
\mathbf n$ goes to infinity, where $W$ is a Gaussian process (or a mixture 
of a Gaussian process).
Usually, the normalizing term $a_{\mathbf n}$ will be chosen 
as $\abs{\mathbf n}:=\prod_{q=1}^dn_q$.

The question of the functional central limit theorem in the space of 
continuous functions (endowed with the uniform norm) for strictly stationary 
random fields has been studied. Wichura \cite{MR0246359} established 
such a result for an i.i.d. centered random field with finite variance, 
which generalized Donsker's one dimensional result \cite{MR0040613}.
Wichura's result was extended to a class of stationary ergodic martingale 
differences random fields \cite{MR542479,MR1629903}, and Dedecker 
found a projective condition \cite{MR1875665}. Wang and Woodroofe
\cite{MR3222815} attempted to extend the Maxwell and Woodroofe condition 
\cite{MR1782272} but found a weaker condition, which was improved 
by Voln\'y and Wang \cite{MR3264437}. The latter is a multidimensional 
extension of Hannan's condition \cite{MR0331683}.
In the context of the mentioned works, the limiting process is a standard Brownian sheet 
when the considered random field is ergodic, that is, a Gaussian process 
$\left(W_{\mathbf t}\right)_{\mathbf t\in [0,1]^d}$ such that 
$\operatorname{Cov}\left(W_{\mathbf t};W_{\mathbf s}\right)=
\prod_{i=1}^d\min\ens{t_i,s_i}$.

\subsection{Orthomartingales}

Let $\pr{T_q}_{q=1}^d$ be bijective, bi-measurable and measure preserving 
transformations on $\pr{\Omega,\mathcal F,\mu}$. Assume that $T_q\circ T_{q'}=
T_{q'}\circ T_q$ for each $q,q'\in \ens{1,\dots,d}$. Let $\f_{\gr{0}}$ be a 
sub-$\sigma$-algebra of $\mathcal F$ such that for each $q\in \ens{1,\dots,d}$, 
$\f_{\gr{0}}\subset T_q^{-1}\f_{\gr{0}}$. 
In this way, 
$\mathcal F_{\mathbf i}:=T^{-\mathbf i}\f_{\gr{0}}$, $\mathbf i\in \Z^d$, yields a filtration. If 
for each $\mathbf k,\mathbf l\in \Z^d$ and each integrable and
$\mathcal F_{\mathbf l}$-mesurable random 
variable $Y$, 
 \begin{equation}
  \mathbb E\left[Y\mid \mathcal F_{\mathbf k}\right]=
  \mathbb E\left[Y\mid \mathcal F_{\mathbf k\wedge \mathbf l}\right] \mbox{ almost surely},
 \end{equation}
the transformations $\pr{T_q}_{q=1}^d$ are said to be \textit{completely 
commuting}.

Recall that $\mathbf i\preccurlyeq \mathbf j$ means that $i_q\leqslant j_q$ for
  each $q\in \ens{1,\dots,d}$.
The collection of random variables $\ens{M_{\mathbf n},\mathbf n\in \N^d}$ is 
said to be an orthomartingale random field with respect to the completely
commuting filtration 
$\left(T^{-\mathbf i}\f_{\gr{0}}\right)_{\mathbf i\in \Z^d}$ if 
for each $\mathbf n\in \N^d$, $M_{\mathbf n}$ is 
$\mathcal F_{\mathbf n}$-measurable, integrable 
and for each  $\mathbf i,\mathbf j\in \Z_+^d$ such that  $\mathbf i\preccurlyeq \mathbf j$, 
  \begin{equation}
   \mathbb E\left[M_{\mathbf j}\mid\mathcal F_{\mathbf i}\right]=M_{\mathbf i}.
  \end{equation}

\begin{Definition}
 Let $m\colon\Omega\to \R$ be a measurable function.  
 The random field $\left(m\circ T^{\mathbf i}\right)_{\mathbf i\in \Z^d}$ is an 
 orthomartingale difference random field with respect to the completely commuting 
 filtration $\pr{T^{-\mathbf i}\f_{\gr{0}}}_{\gr{i}\in \Z^d}$ if the random field 
 $\left(M_{\mathbf n}\right)_{\mathbf n\in \N^d}$ defined by 
 $M_{\mathbf n}:=\sum_{\mathbf i\in [\mathbf 0, \mathbf n-\mathbf 1]}
 m\circ T^{\mathbf i}$ is an orthomartingale random field.
\end{Definition}

  \begin{Proposition}\label{prop:Doob}
   Let $\pr{m\circ T^{\mathbf i}}_{\mathbf i\in \Z^d}$ be an orthomartingale differences 
   random field with respect to the completely commuting filtration 
   $\pr{T^{-\mathbf i}\f_{\gr{0}}}_{\mathbf i\in \Z^d}$. Then for each 
   $\mathbf n\in \N^d$ such that $\mathbf n\succcurlyeq \mathbf 1$, the following inequality 
   holds:
   \begin{equation}\label{eq:Doob}
    \norm{\frac 1{\abs{\gr{n}}^{1/2}} 
    \max_{\mathbf 1\preccurlyeq \mathbf i\preccurlyeq\mathbf n}
    \abs{S_{\mathbf i}\pr{m}}
    }\leqslant 2^d\norm{m}.
   \end{equation}

  \end{Proposition}
  This shows that the the family of normalized maxima of partial sums is
  bounded in $\mathbb L^2$. Lemma~3.1 in \cite{MR3264437} shows more.
  
  \begin{Proposition}\label{prop:uniform_integrability}
   Let $\pr{m\circ T^{\mathbf i}}_{\mathbf i\in \Z^d}$ be an orthomartingale differences 
   random field with respect to the completely commuting filtration 
   $\pr{T^{-\mathbf i}\f_{\gr{0}}}_{\mathbf i\in \Z^d}$. Then the family 
   \begin{equation}
    \ens{\frac 1{\abs{\mathbf n}} \max_{\mathbf 1\preccurlyeq \mathbf i 
    \preccurlyeq\mathbf n}\abs{S_{\mathbf i}\pr{m}}^2 , \mathbf n\in \pr{\N^*}^d
    } 
   \end{equation}
  is uniformly integrable.
  \end{Proposition}

  \subsection{Orthomartingale approximation}
  
  There are essentially two methods for establishing the 
  invariance principle for a stationary sequence. The first one 
  is approximation by an i.i.d. sequence, which leads to good results 
  but there are processes which cannot be treated in this way.
  An other method for establishing limit theorems
  for strictly stationary sequences is a martingale approximation. 
  Since it is known that a stationary martingale difference sequence 
  satisfies the invariance principle, one can try to prove an invariance principle 
  by martingale approximation. More formally, given a square integrable 
  centered function $f\colon\Omega\to \R$, 
  one can wonder whether there exists a square integrable martingale differences sequence 
  $\pr{m\circ T^i}_{i\geq 0}$ such that $\lim_{n\to +\infty}
  n^{-1/2}\norm{\max_{1\leq i\leq n}\abs{S_i\pr{f-m}}}=0$.
  The existence of such an approximation without the $\max$ has 
  been investigated in \cite{MR0251785}. A necessary and sufficient 
  condition has been given in \cite{MR2060314,MR2462550} in the adapted case, then 
  extended to the nonadapted case in \cite{MR2283254}. The question of 
  the choice of filtration has also been considered in \cite{MR2914436}.
  
  This approach was also used for other limit theorems, like the 
  quenched weak invariance principle \cite{MR3083921,MR3178473} or
  the law of the iterated logarithms \cite{MR2370600}.

  A multidimensional analogue of the martingale approximation has not 
  been so intensively studied. There are various way to define martingales 
  random fields in dimension greater than one (cf. \cite{MR1162153,MR0254912}). 
  
  In this paper, we shall work on orthomartingale approximation, since it is 
  known \cite{MR3427925,MR3504508} that when $T_1$ is ergodic, 
  the invariance principle takes place.

  \begin{Definition}
   We say that the function $f$ admits an orthomartingale approximation if 
   there exists a square integrable function $m$ such 
   that $\pr{m\circ T^{\gr{i}}}_{\gr{i}\in \Z^d}$ is 
an orthomartingale differences random field
   \begin{equation}\label{eq:def_orthomartingale_approximation}
 \limsup_{\mathbf n\to +\infty}\frac 1{\sqrt{\abs{\mathbf n}}}
 \norm{\max_{\mathbf 1\preccurlyeq \mathbf i\preccurlyeq \mathbf n}\abs{
 S_{\mathbf i}(f-m)}}=0.
\end{equation}
  \end{Definition}
  
  The uniform norm of the function $t\mapsto S_{\gr{n}}\pr{f,t}$ 
  can be controlled by the maxima of partial sums. Moreover, 
  a stationary orthomartingale differences random field with 
  respect to a completely commuting filtration 
  such that one of the maps $T_1,\dots,T_d$ is ergodic 
  satisfies the invariance principle. Therefore, when 
  ergodicity in one direction holds, an orthomartingale approximation 
  entails the invariance principle. In the other cases, an invariance principle 
  may still hold, but the limiting process may not be a Brownian sheet 
  (see Remark 5.5 in \cite{MR3222815}).
  
  The paper is organizes as follows. Section~\ref{sec:main_results} contains 
  the main results of the 
  paper, namely, a necessary and sufficient condition for the 
  orthomartingale-coboundary decomposition, a sufficient condition 
  for the existence of an approximating orthomartingale and the 
  verification of the latter under two projective condition: Hannan 
  and Maxwell-Woodroofe. Section~\ref{sec:proofs} is devoted the 
  proofs.

\section{Main results}\label{sec:main_results}
 
\subsection{Orthomartingale-coboundary decomposition}
 
The following operators will be used in the sequel.

 \begin{Definition}\label{dfn:contractions}
   Let $T$ be a measure preserving $\Z^d$-action and let $\f_{\gr{0}}$ be 
   a $\sigma$-algebra such that $\pr{T^{-\gr{i}}\f_{\gr{0}}}_{\mathbf i
   \in \Z^d}$ is a completely commuting filtration.  
   Let $E\subsetneq [d]$ and $\gr{i}\in \N^d$. We define the 
   operators $P_{d,E}$ and $P_{d,[d]}$ by
   \begin{equation}\label{eq:definition_de_Pe}
    P_{d,E}^{\gr{i}}\pr{f}:=
    \sum_{J\subset E}\pr{-1}^{\abs{J}+\abs{E}} 
    \E{U^{\gr{i} \cdot \eps\pr{E}} f\mid \f_{\infty \gr{1_J}}}, \quad f\in 
    \mathbb L^1,
   \end{equation}
   \begin{equation}\label{eq:definition_de_Pd}
    P_{d,[d]}^{\gr{i}}\pr{f}=
    U^{-\gr{i}}f+\sum_{J\subsetneq [d]}\pr{-1}^{\abs{J}+d}
    \E{U^{-\gr{i}}f\mid \f_{\infty \gr{1_J}}}, \quad f\in 
    \mathbb L^1,
   \end{equation}
   and the closed subspaces of $\mathbb L^2$
   \begin{equation}\label{eq:definition_de_He}
    \h_{d,E}:=\ens{h\in \mathbb L^2_0\mid  
    h\mbox{ is }\f_{\infty \gr{1_E}}\mbox{-measurable and }
    \E{h\mid \f_{\infty \gr{1_{E'}}}}=0\mbox{ if }
    E'\subsetneq E}, E\subsetneq [d],
   \end{equation}
  \begin{equation}\label{eq:definition_de_Hd}
   \h_{d,[d]}=\ens{h\in \mathbb L^2\mid  \E{h\mid \f_{\infty \gr{1_{E'}}}}=0\mbox{ if }
    E'\subsetneq [d]}.
  \end{equation}

  \end{Definition}

  When the integer $d$ does not need to be specified, we shall simply
  denote $P_E^{\gr{i}}$ for $E\subset [d]$ and $\h_E$.
  
  In dimension one, we have 
  \begin{equation}
   P_{\emptyset}^i\pr{f}:=\E{U^if\mid \f_0}\mbox{ and }
   P_{\ens{1}}^i\pr{f}:=U^{-i}f-\E{U^{-i}f\mid \f_0},
  \end{equation}
  and these operators have been used in 
  \cite{MR2344817,MR3178617,1510.01459}.
  In dimension two, the operators $P_{2,E}$ are given by 
  \begin{equation}
   P_{\emptyset}^{i,j}\pr{f}=\E{U^{\pr{i,j}}f\mid \f_{0,0}},
  \end{equation}
  \begin{equation}
   P_{\ens{1}}^{i,j}\pr{f}=\E{U^{\pr{-i,j}}f\mid \f_{\infty,0}}
   -\E{U^{\pr{-i,j}}f\mid \f_{0,0}},
  \end{equation}
  \begin{equation}
   P_{\ens{2}}^{i,j}\pr{f}=\E{U^{\pr{i,-j}}f\mid \f_{0,\infty}}
   -\E{U^{\pr{i,-j}}f\mid \f_{0,0}}\mbox{ and }
  \end{equation}
  \begin{multline}
   P_{\ens{1,2}}^{i,j}\pr{f}=U^{\pr{-i,-j}}f
   -\E{U^{\pr{-i,-j}}f\mid \f_{\infty,0}}\\
   -\E{U^{\pr{-i,-j}}f\mid \f_{0,\infty}}
   +\E{U^{\pr{-i,-j}}f\mid \f_{0,0}}.
  \end{multline}
  We are now in position to state a necessary and 
  sufficient condition for the orthomartingale-coboundary 
  decomposition.
  
  \begin{Theorem}\label{thm:OMC_decomposition}
   Let $f$ be a square integrable centered function and $d\geq 1$. 
   Let $T$ be a measure preserving $\Z^d$-action and let $\f_{\gr{0}}$ be 
   a $\sigma$-algebra such that $\pr{T^{-\gr{i}}\f_{\gr{0}}}_{\mathbf i
   \in \Z^d}$ is a completely commuting filtration. The following conditions 
   are equivalent:
   \begin{enumerate}
    \item for each $E\subset [d]$, 
    \begin{equation}\label{eq:CNS_for_orthomartingale_approximation}
     \sup_{\gr{n}\in \N^d}\norm{ 
     \sum_{\gr{0}\imd \gr{i}\imd\gr{n}}P_E^{\gr{i}}f
     }<+\infty;
    \end{equation}
   \item there exists square integrable functions $m_J$, $J\subset [d]$ such that 
   \begin{equation}
    f=\sum_{J\subset [d]}\prod_{q\in  J}\pr{\Id-U_q}m_J
   \end{equation}
  and for each $J\neq [d]$, $m_J$ is $\f_{\infty   \gr{1_J}}$-measurable 
  and if $I\subsetneq J$, then $\E{m_J\mid \f_{\infty   \gr{1_I}}}=0$.
   \end{enumerate}

  \end{Theorem}
  
  \begin{Remark}
   In dimension one, Theorem~\ref{thm:CNS_orthomartingale_approximation} reads as follows:
   a function $f$ can be written as $f=m+g-g\circ T$, where $m$ is $\f_0$-measurable 
   and $\E{m\mid T\f_0}=0$ if and only if 
   \begin{equation}
    \sup_{n\geq 0}\norm{\E{S_n\pr{f}\mid\f_0}}<+\infty\mbox{ and }
    \sup_{n\geq 0}\norm{S_n\pr{f}-\E{S_n\pr{f}\mid T^{-n+1}\f_0}}<+\infty
   \end{equation}
  This can be viewed as a nonadapted version
  of Proposition~4.1 in \cite{MR3178617}.
  \end{Remark}
 
 \begin{Remark}
  This improves the main result in \cite{MR3522451} since 
  Theorem~\ref{thm:OMC_decomposition} does not require the 
  function $f$ to be $\f_{\gr{0}}$-measurable. Moreover, 
  even for such functions, the condition is less restrictive. Indeed, 
  in this case, 
  condition \eqref{eq:CNS_for_orthomartingale_approximation} is equivalent 
  to boundedness of the quantity $\norm{ 
  \E{S_{\gr{n}}\pr{f}\mid\f_{\gr{0}}}}$ independently of $\gr{n}\smd \gr{1}$, 
  while that of \cite{MR3522451} read $\sum_{\gr{k}\in \N^d}
  \E{U^{\gr{k}}f\mid\f_{\gr{0}}}<+\infty$.
  
  A related result has been obtained in \cite{MR2749126}, where 
  reversed martingales are obtained in the decomposition. However, 
  a kind of regularity assumption is made. This is also the case in 
  \cite{Volny2017}.
 \end{Remark}

\subsection{A sufficient condition for orthomartingale approximation}  
 
In order to express a sufficient condition for the 
orthomartingale approximation \eqref{eq:def_orthomartingale_approximation}, 
we define a blocking operator. 

\begin{Definition}
 Let $f\colon \Omega\to \R$ be a measurable function and let $k$ be an 
 integer greater or equal to $1$. The \emph{blocking operator} is 
 defined by 
 \begin{equation}\label{eq:def_blocking_operator}
  B_k(f):=\frac 1{k^d}\sum_{\mathbf 1\preccurlyeq \mathbf i 
  \preccurlyeq k\mathbf 1}\sum_{E\subset [d]}  
  \prod_{q=1}^d\pr{\Id-P_E^{i_q\gr{e_q}}}P_E^{\gr{0}}\pr{f}.
 \end{equation}
\end{Definition}

\begin{Definition}
 Let $f\colon\Omega\to \R$ be a measurable function. The \emph{plus semi-norm}, 
 denoted by $\norm{\cdot}_+$, is defined by 
 \begin{equation}\label{eq:definition_norm_plus}
  \norm{f}_+:=\limsup_{\mathbf n\to +\infty}
  \frac 1{\abs{\mathbf n}^{1/2}}\norm{
  \max_{\mathbf 1\preccurlyeq \mathbf i\preccurlyeq\mathbf n}
  \abs{S_{\mathbf i}(f)}}.
 \end{equation}

\end{Definition}

\begin{Theorem}\label{thm:CNS_orthomartingale_approximation}
Let $\pr{\Omega,\mathcal F,\mu}$ be a probability space and let 
$T\colon \Omega\to\Omega$ be a measure preserving $\Z^d$-action. 
Assume that $\f_{\gr{0}}$ is a sub-$\sigma$-algebra of $\f_{\gr{0}}$ such that 
$\pr{T^{-\mathbf i}\f_{\gr{0}}}_{\mathbf i \in \Z^d}$ is a completely 
commuting filtration. Let $f\colon \Omega\to \R$ be a measurable function.  
If 
\begin{equation}\label{eq:sufficient_cond_MA}
 \lim_{k\to +\infty}\norm{B_k(f)-f}_+=0,
\end{equation}
then there exists a function $m$ such that 
$\pr{m\circ T^{\gr{i}}}_{\gr{i}\in \Z^d}$ is an orthomartingale 
differences random field with respect to the filtration 
$\pr{T^{-\mathbf i}\f_{\gr{0}}}_{\mathbf i \in \Z^d}$ and
\begin{equation}\label{eq:def_orthomartingale_approximation_thm}
 \limsup_{\mathbf n\to +\infty}\frac 1{\sqrt{\abs{\mathbf n}}}
 \norm{\max_{\mathbf 1\preccurlyeq \mathbf i\preccurlyeq \mathbf n}\abs{
 S_{\mathbf i}\pr{f-m}}}=0.
\end{equation}
In particular, the conclusion holds if \eqref{eq:sufficient_cond_MA} is 
replaced by the following one:
\begin{equation}\label{eq:easier_sufficient_cond}
 \forall \emptyset\subsetneq J\subset [d],\forall 
 E\subset [d], \quad \lim_{k\to +\infty}\frac 1{k^{\abs{J}}}\norm{ 
 \sum_{\gr{1_J}\imd \gr{j}\imd k\gr{1_J}}
 P_E^{\gr{j}}f
 }_+=0.
\end{equation}

\end{Theorem}

\begin{Remark}
 In dimension one, Theorem~\ref{thm:CNS_orthomartingale_approximation} reads as follows: 
 the condition 
 \begin{equation}
  \lim_{k\to +\infty}\frac 1k\limsup_{n\to+\infty}
  \frac 1{\sqrt n}\norm{\max_{1\leqslant i\leqslant n}\abs{S_i 
 \pr{\mathbb E\left[ S_k(f)\mid \f_{\gr{0}}\right] +
  \sum_{i=1}^kU^{-i}f-\E{U^{-i}f\mid \f_0}     }
  }}=0.
 \end{equation}
 is  sufficient for the existence of a function $m$ such that 
 $\pr{m\circ T^i}_{i\geqslant 0}$ is a martingale differences sequence 
 and $\limsup_{n\to+\infty}
 n^{-1/2}\norm{\max_{1\leqslant i\leqslant n}\abs{S_i\pr{f-m}}}=0$. This can 
 be viewed as a nonadapted version of Theorem~1 in \cite{MR2797997}.
\end{Remark}

\begin{Remark}
 Theorem~\ref{thm:CNS_orthomartingale_approximation} can be used even if 
 none of the maps $T_q$, $q\in [d]$ is ergodic. In this case, the function $f$ 
 may not satisfy the central limit theorem because the approximating martingale 
 itself may not satisfy it (see Remark~5.5 in \cite{MR3222815}).  
\end{Remark}

\begin{Remark}
 In all the paper, we assume the filtration to be completely commuting. In 
 \cite{2017arXiv170201143P}, partially commuting filtration $\pr{\f_{\gr{i}}}_{\gr{i}
 \in \Z^d}$ are considered, in the sense that if $i\geq i'$ and 
 $\gr{u},\gr{v}\in \Z^{d-1}$, then for any integrable random variable $Y$, 
 \begin{equation}\label{eq:partially_commuting}
  \E{ 
  \E{
  X\mid \f_{i',\gr{v}}
  }\mid 
  \f_{i,\gr{u}}
  }=\E{X\mid \f_{i',\min\ens{\gr{u},\gr{v}}}}.
 \end{equation}
It does not seem that our results apply in this context because 
complete commutativity of the filtration is used in the proof of 
Theorem~\ref{thm:CNS_orthomartingale_approximation}.
\end{Remark}

  \subsection{Applications: projective conditions}
  
  \subsubsection{Hannan's condition}
 
 Assume that $d=1$, $T\colon \Omega\to \Omega$ is a bijective 
 bimeasurable measure preserving map and $\f_0$ is a sub-$\sigma$-algebra 
 such that $T\f_0\subset\f_0$. Assume that $f\colon \Omega\to \R$ 
 is measurable with respect to the $\sigma$-algebra generated by 
 $\bigcup_{k\in \Z}T^k\f_0$ and such that $\E{f\mid \bigcap_{k\in \Z}
 T^k\f_0}=0$ and let us consider the condition
 \begin{equation}\label{eq:Hannan_dim1}
  \sum_{i\in \Z}\norm{\E{f\circ T^i\mid \f_0}
  -\E{f\circ T^i\mid T\f_0}}<+\infty.
 \end{equation}
 That the central limit theorem is implied by \eqref{eq:Hannan_dim1}
 is contained in \cite{MR0372955} (see also Theorem~6 in 
 \cite{MR1198662}).
 When $f$ is $\f_0$-measurable, the central limit theorem and 
 the weak invariance principle were proved by Hannan 
 \cite{MR0331683,MR562049} under the assumption that $T$ 
 is weakly mixing. Dedecker and Merlevède \cite{MR2019054} 
 showed that \eqref{eq:Hannan_dim1} itself implies the 
 weak invariance principle. Finally, the invariance principle 
 when 
 $f$ satisfies \eqref{eq:Hannan_dim1} but is not necessarily 
 $\f_0$-measurable was establised in \cite{MR2359065}.

 The generalization of condition \eqref{eq:Hannan_dim1} 
 to random field has been obtained by Voln\'y and Wang.
  Let us recall the notations and results of \cite{MR3264437}.
 The projection operators with respect to a commuting filtration 
$\left(\mathcal F_{\mathbf i}\right)_{\mathbf i\in \Z^d}$ are 
defined by 
  
  \begin{equation}\label{eq:definition_projectors}
 \pi_{\mathbf j}:=\prod_{q=1}^d\pi_{j_q}^{(q)},\quad \mathbf j\in \Z^d,
\end{equation}
where for $l\in \Z$, $\pi_l^{(q)}\colon \mathbb L^1(\mathcal F)\to \mathbb L^1(\mathcal F)$ 
is defined for $f\in\mathbb L^1$ by 
\begin{equation}
 \pi_l^{(q)}(f)=\mathbb E_l^{(q)}\left[f\right]-\mathbb E_{l-1}^{(q)}\left[f\right]
\end{equation}
and 
\begin{equation}
 \mathbb E_l^{(q)}\left[f\right]=\mathbb E\left[f\mid 
 \bigvee_{\mathclap{\substack{\mathbf i\in \Z^d \\
 i_q\leqslant l}}}\mathcal F_{\mathbf i}\right], q\in [d],l\in \Z.
\end{equation}

  \begin{Theorem}[\cite{MR3264437}]\label{prop:Hannan}
   Let $\pr{\f_{\gr{i}}}_{\gr{i}\in\Z^d}:=
   \pr{T^{-\mathbf i}\f_{\gr{0}}}_{\mathbf i\in \Z^d}$ be a completely commuting 
   filtration. 
   Let $f$ be a function such that for each $q\in [d]$, $\E{ 
   f\mid T_q^l\f_{\gr{0}} }\to 0$ as $l\to +\infty$, measurable with respect to the 
   $\sigma$-algebra generated by $\bigcup_{\mathbf i\in\Z^d}T^{\mathbf i}
   \f_{\gr{0}}$ and such that $\sum_{\mathbf i\in \Z^d}\norm{\pi_{\mathbf i}(f)}
   <+\infty$. Then there exists a function $m$ such that $\pr{m\circ 
   T^{\mathbf i}}_{\mathbf i\in \Z^d}$ is an orthomartingale differences random 
   field with respect to the 
   completely commuting filtration $\pr{T^{-\mathbf i}\f_{\gr{0}}}_{\mathbf i\in \Z^d}$ and 
   such that \eqref{eq:def_orthomartingale_approximation} holds. 
  \end{Theorem}
  
  We can recover this result via Theorem~\ref{thm:CNS_orthomartingale_approximation}.

  \subsubsection{Maxwell and Woodroofe condition}
  
   In the one dimensional case, conditions on the 
   quantities $\E{S_n(f)\mid T\f_{0}}$ 
  and $S_n(f)-\E{S_n(f)\mid T^{-n}\f_{0}}$ have been investigated.
  The first result in this direction was obtained by Maxwell and Woodroofe 
  \cite{MR1782272}: if $f$ is $\f_{0}$-measurable and 
  \begin{equation}\label{MW_adapted}
   \sum_{n=1}^{+\infty}\frac{\norm{\mathbb E\left[S_n(f)\mid \f_{0}\right]}}{n^{3/2} }
   <+\infty,     
  \end{equation}
  then $\pr{n^{-1/2}S_n(f) }_{n\geqslant 1}$ converges in distribution to $\eta^2 N$, 
  where $N$ is normally distributed and independent of $\eta$. 
  Then Voln\'y \cite{MR2283254} proposed a method to treat the nonadapted case. 
  Peligrad and Utev \cite{MR2123210} proved the weak invariance principle 
  under condition \eqref{MW_adapted}. The nonadapted case was addressed 
  in \cite{MR2344817}. Peligrad and Utev also showed that condition \eqref{MW_adapted}
  is optimal among conditions on the growth of the sequence 
  $\left(\norm{ \mathbb E\left[S_n(f)\mid \f_{0} \right]}\right)_{n\geqslant 1}$:
  if 
  \begin{equation}\label{weaked_MW}
   \sum_{n=1}^{+\infty}a_n\frac{\norm{\mathbb E\left[S_n(f)\mid \f_{0}\right]}}{n^{3/2} }<\infty
  \end{equation}
 for some sequence $\pr{a_n}_{n \geqslant 1}$ converging to $0$, 
 the sequence $\pr{n^{-1/2}S_n(f)}_{n \geqslant 1}$ 
 is not necessarily stochastically bounded (Theorem~1.2. of \cite{MR2123210}). 
 Voln\'y constructed  \cite{MR2679961} an example satisfying \eqref{weaked_MW} and 
 such that the sequence $\left(\norm{S_n(f)}^{-1}S_n(f)\right)_{n\geqslant 1}$ 
 admits two subsequences which converge weakly to two different distributions. 
 In dimension one, these results are the consequence of a existence of an 
 approximating martingale (see Proposition~3 in \cite{MR2797997}).
 We are able to formulate an analoguous result in the multidimensional 
 setting.
  
  \begin{Theorem}\label{thm:approximation_sous_Maxwell_Woodroofe}
  Let $T$ be a measure preserving $\Z^d$-action and let 
  $\f_{\gr{0}}$ be a sub-$\sigma$-algebra such that $\pr{T^{-\mathbf i}
  \f_{\gr{0}}}_{\mathbf i\in \Z^d}$ is a completely commuting filtration.
  Let $f$ be a square integrable function such that for any $E\subset [d]$, 
  \begin{equation}\label{eq:MW_dim_d}
   \sum_{\mathbf n\succcurlyeq \mathbf 1}
   \frac 1{\abs{\mathbf n}^{3/2}} 
   \norm{ \sum_{\gr{0}\imd\gr{i}\imd\gr{n}-\gr{1} } P_E^{\gr{i}}f }<+\infty.
  \end{equation}
  Then there exists a function $m$ such that 
  $\pr{m\circ T^{\gr{i}}}_{\gr{i}\in \Z^d}$ is 
an orthomartingale differences random field and
   \begin{equation}
 \limsup_{\mathbf n\to +\infty}\frac 1{\sqrt{\abs{\mathbf n}}}
 \norm{\max_{\mathbf 1\preccurlyeq \mathbf i\preccurlyeq \mathbf n}\abs{
 S_{\mathbf i}\pr{f-m}}}=0.
\end{equation}

  \end{Theorem}
\begin{Remark}
 In dimension one, we recover the result of \cite{MR2123210,MR2344817}. 
 In dimension two, condition \eqref{eq:MW_dim_d} reads as follows: if 
 the series 
 \begin{equation}
 A_{\emptyset}:= \sum_{n_1,n_2\geq 1}
  \frac 1{n_1^{3/2}n_2^{3/2}}
  \norm{\E{S_{n_1,n_2}\pr{f}\mid \f_{0,0}}}
 \end{equation}
\begin{equation}
 A_{\ens{1}}:=
 \sum_{n_1,n_2\geq 1}
  \frac 1{n_1^{3/2}n_2^{3/2}}\norm{\E{S_{n_1,n_2}\pr{f}\mid \f_{\infty,0}}-
  \E{S_{n_1,n_2}\pr{f}\mid \f_{n_1-1,0}}} 
\end{equation}
\begin{equation}
 A_{\ens{2}}:=
 \sum_{n_1,n_2\geq 1}
  \frac 1{n_1^{3/2}n_2^{3/2}}
  \norm{\E{S_{n_1,n_2}\pr{f}\mid \f_{0,\infty}}-
  \E{S_{n_1,n_2}\pr{f}\mid \f_{0,n_2-1}}} 
\end{equation}
\begin{multline}
 A_{\ens{1,2}}:=
\sum_{n_1,n_2\geq 1}
  \frac 1{n_1^{3/2}n_2^{3/2}}
  \left\lVert S_{n_1,n_2}\pr{f}-\E{S_{n_1,n_2}\pr{f}\mid \f_{n_1-1,\infty}}\right.\\
  \left.-\E{S_{n_1,n_2}\pr{f}\mid \f_{\infty,n_2-1}} +
  \E{S_{n_1,n_2}\pr{f} \mid \f_{n_1-1,n_2-1}} \right\rVert
\end{multline}
are convergent, then there exists an orthomartingale differences 
random fields satisfying \eqref{eq:def_orthomartingale_approximation}. If $f$ is 
$\mathcal F_{\gr{0}}$-measurable, then the series $A_{\ens{1}}$, $A_{\ens{2}}$ and 
$A_{\ens{1,2}}$ are convergent.
\end{Remark}

\begin{Remark}
 Using an adaptation of the construction given in \cite{MR2446326,MR2475604}, we 
 can construct an example of function $f$ which satisfies the assumption
 of Proposition~\ref{prop:Hannan} but not that of 
 Theorem~\ref{thm:approximation_sous_Maxwell_Woodroofe} and vice-versa. 
 Let  $\pr{\Omega_1,\mathcal A_1,\mu_1,T_1}$ be the dynamical system 
 considered in \cite{MR2446326,MR2475604} and for 
 $2\leq i\leq d$, let $\pr{\Omega_i,\mathcal A_i,\mu_i,T_i}$ be 
 Bernoulli dynamical systems. Then consider $\Omega:=\prod_{q=1}^d\Omega_q$, 
 $\mathcal A:=\bigotimes_{q=1}^d\mathcal A_q$, $T^\gr{i}
 \pr{\pr{\omega_q}_{q=1}^d}:=\pr{T_q^{i_q}\omega_q}_{q=1}^d$. For 
 $2\leq q\leq d$, let $e^{(q)}\colon \Omega_q\to \R$ be such that 
 $\pr{e^{(q)}\circ T_q^i}_{i\in \Z}$ is i.i.d. If $f$ is the function 
 defined in \cite{MR2446326,MR2475604}, then let $F=f\cdot\prod_{q=2}^d
 e^{(q)}$ and $\f_{\gr{0}}:=\f_0\otimes\bigotimes_{q=2}^d\sigma\pr{ 
 e^{(q)}\circ T_q^i,i\leq 0
 }$. In this way, the $F$ satisfies the multidimensional Hannan and Maxwell and 
 Woodroofe conditions if and only if so does $f$ for the unidimensional 
 ones.
\end{Remark}
\begin{Remark}
 Using the same construction as previously, but where 
 $\pr{\Omega_1,\mathcal A_1,\mu_1,T_1}$ is the dynamical system 
 involved in the proof of Theorem~\cite{MR2123210}, we can see that 
 the weight $\abs{\gr{n}}^{-3/2}$ in condition \eqref{eq:MW_dim_d} 
 is in some sense optimal.
\end{Remark}

\begin{Remark}
 If $f$ is an $\f_{\gr{0}}$-measurable function, then condition 
 \eqref{eq:MW_dim_d} holds as soon as 
 \begin{equation}\label{eq:weakned_MW}
  \sum_{\gr{n}\smd \gr{1}} \frac 1{\abs{\gr{n}}^{1/2}}
  \norm{\E{U^{\gr{n}} \pr{f}\mid \f_{\gr{0}}   }}<+\infty.
 \end{equation}
It was proven in \cite{MR3222815} that when the filtration 
$\pr{\f_{\gr{i}}}_{\gr{i}\in \Z^d}$ is generated by an i.i.d. random field, 
condition \eqref{eq:weakned_MW} implies the central limit theorem. Moreover, 
if \eqref{eq:weakned_MW} holds when the $\mathbb L^2$-norm is replaced by 
the $\mathbb L^p$-norm for some $p>2$, then the invariance principle holds. 
Our result thus extend these ones, since only a finite moment of order two 
is required and the condition on the dependence is weaker. 
If  $f$ is a function such that for each $q\in [d]$, $\E{ 
   f\mid T_q^l\f_{\gr{0}} }\to 0$ as $l\to +\infty$, measurable with respect to the 
   $\sigma$-algebra generated by $\bigcup_{\mathbf i\in\Z^d}T^{\mathbf i}
   \f_{\gr{0}}$ and satisfies \eqref{eq:weakned_MW}, then 
   by Lemma~6.2 in \cite{MR3264437}, $\sum_{\mathbf i\in \Z^d}\norm{\pi_{\mathbf i}(f)}
   <+\infty$.
\end{Remark}
\begin{Remark}
 In \cite{2017arXiv170201143P}, a central limit theorem has been obtain 
 for an $\f_{\gr{0}}$-measurable function $f$ satisfying 
 \eqref{eq:MW_dim_d}. Their result applies in the context 
 of partially commuting filtrations (see \eqref{eq:partially_commuting}), 
 which includes a larger class of filtrations than completely commuting 
 ones. Nevertheless, our results include the nonadapted case and 
 lead to an invariance principle.
\end{Remark}

\begin{Remark}
Condition \eqref{eq:MW_dim_d} is much less restrictive than admitting an orthomartingale 
coboundary decomposition in $\mathbb L^2$ (see Theorem~\ref{thm:OMC_decomposition}). 
\end{Remark}

  A key step for proving that a function satisfying the Maxwell and 
  Woodroofe condition also satisfies the conditions of 
 Theorem~\ref{thm:CNS_orthomartingale_approximation} is a maximal inequality, which is 
  of independent interest. Note that a similar inequality has been obtained 
  in \cite{MR3222815} but without the maxima.

  \begin{Proposition}\label{prop:inegalite_maximale}
  Let $d\geqslant 1$ be an integer. There exists a constant $C(d)$ such that 
  for any $\Z^d$-measure preserving action $T$, any $\mathcal F_{\gr{0}}$
  such that $\pr{T^{-\mathbf i}\f_{\gr{0}}}_{\mathbf i\in \Z^d}$ be a completely commuting 
   filtration, any $\mathbf n\succcurlyeq \gr{0}$, any $E\subset [d]$ and any
   $f\in \h_E$:
  \begin{multline}\label{eq:inegalite_maximale_enonce}
   \norm{\max_{\mathbf 1\preccurlyeq\mathbf i\preccurlyeq \mathbf 2^{\mathbf n}}
   \abs{S_{\mathbf i}(f)}
   }\leqslant C(d) \abs{2^{\gr{n}}}^{1/2}  
  \sum_{\mathbf 0\preccurlyeq \mathbf i\preccurlyeq \mathbf n}
  2^{-\mathbf i/2}\norm{\sum_{\gr{0}\imd\gr{j}\imd \gr{2^i} } 
  P_E^{\gr{j}} f     } \\ 
  \leqslant C(d)^2  \abs{2^{\gr{n}}}^{1/2}  
  \sum_{\gr{n}\smd \gr{1}}\frac 1{\abs{\gr{n}}^{3/2}}
  \norm{\sum_{\gr{0}\imd\gr{j}\imd \gr{n}-\gr{1} } 
  P_E^{\gr{j}} f     }.
  \end{multline} 
  
  \end{Proposition}

 Examples 5 and 6 in \cite{2017arXiv170201143P} are formulated in the 
 context of completely commuting filtration. Our results can be used to 
 treat non causal linear and Volterra random fields. 
 We derive from Theorem~\ref{thm:approximation_sous_Maxwell_Woodroofe} 
 a sufficient condition for a linear random field to satisfy the 
 weak invariance principle.
 
 \begin{Corollary}\label{cor:linear_processes}
  Let $\pr{\eps_{\gr{i}}}_{\gr{i}\in \Z^d}$ be 
  an i.i.d. random field where $\eps_{\gr{0}}$ is 
  centered and square integrable. Let 
  $T^{\gr{j}}\colon \R^{\Z^d}\to \R^{\Z^d}$ be defined as 
  $T^{\gr{j}}\pr{\pr{x_{\gr{i}}}_{\gr{i}\in \Z^d}}
  =\pr{\pr{x_{\gr{i}+\gr{j}}}_{\gr{i}\in \Z^d}}$, $\f_{\gr{0}}=
  \sigma\pr{\eps_{\gr{i}},\gr{i}\imd\gr{0}}$ and 
  \begin{equation}\label{eq:definition_of_linear_process}
   f=\sum_{\gr{i}\in \Z^d}
   a_{\gr{i}} \eps_{-\gr{i}},
  \end{equation}
 where $a_{\gr{i}}\in \R$ and $\sum_{\gr{i}\in \Z^d}a_{\gr{i}}^2
 <+\infty$. Define for $E\subset [d]$ and $\gr{n}\smd\gr{1}$,
 \begin{equation}
  \Delta_{E,\gr{n}}:=
  \sum_{\gr{j}\in \N^d,\gr{j}\cdot \gr{1_E} \smd \gr{1_E} 
  }\pr{\sum_{\gr{0}\imd \gr{k}\imd\gr{n}-\gr{1}} 
  a_{\pr{\gr{k}+\gr{j}}\cdot \eps\pr{E}} }^2.
 \end{equation}
 Assume that for any $E\subset [d]$, the convergence 
 \begin{equation}\label{eq:linear_processes_MW}
  \sum_{\gr{n}\smd\gr{1}}\abs{\gr{n}}^{-3/2}\Delta_{E,\gr{n}}^{1/2}<+\infty
 \end{equation}
 holds. 
 Then $f$ satisfies the invariance principle 
 in $C\pr{\left[0,1\right]^d}$.
 \end{Corollary}

 \begin{Remark}
  Corollary~\ref{cor:linear_processes} also holds 
  when we define the linear process by \eqref{eq:definition_of_linear_process} 
  but the innovations $\eps_{\gr{i}}$ are only supposed to be
  orthomartingale differences with respect to a completely commuting filtration 
  $\pr{\f_{\gr{i}}}_{\gr{i}\in \Z^d}$ (where $\f_{\gr{i}}$ is not supposed to be
  generated by i.i.d.).
 \end{Remark}

\section{Proofs} \label{sec:proofs}
 
 \subsection{Proof of Theorem~\ref{thm:OMC_decomposition}}
  
 \subsubsection{Contractions}
 
 In the next proposition, we collect some properties of the operators 
 $P_{d,E}$ and of the spaces $\h_{d,E}$.
 
 \begin{Proposition}
 \label{prop:Pq}
 
 \begin{enumerate}
  \item \label{itm:contraction} For any $E\subset [d]$ and 
  any square integrable function $f$, $\norm{P_E^{\gr{0}}f}\leq \norm{f}$.
  \item\label{itm:semi_group} Let $d\geq 1$, $E\subset [d]$ 
 and $\gr{j},\gr{k}\in \N^d$. For any 
 function $f\in \h_E$, the function $P_E^{\gr{k}}f$ belongs 
 to $\h_E$ and 
 \begin{equation}\label{eq:semi_group}
  P_E^{\gr{j}}\circ P_E^{\gr{k}}f=P_E^{\gr{j}+\gr{k}}f.
 \end{equation}
  \item\label{itm:commutativity_PU} For any $d\geq 1$ and $E\subset [d+1]$, 
  \begin{equation}
   \widetilde{P}^{\gr{j}}_{d,E\setminus\ens{d+1}} 
   P_{d+1,E}^{\gr{e_{d+1}}}f= 
   P_{d+1,E}^{\gr{j}+\gr{e_{d+1}}}f,
  \end{equation}
 where  $\widetilde{P}=P$ if $d+1\notin E$ and $\widetilde{P}$ 
 is defined similarly as \eqref{eq:definition_de_Pe} and 
 \eqref{eq:definition_de_Pd} but $\f_{\gr{0}}$ is replaced by 
 $\f_{\infty\gr{1_{\ens{d+1}}}}$.

 \item\label{itm:commutativity_P_U_d} For any $d\geq 1$ and 
 any positive integer $j$, 
 \begin{equation}
  P_{d+1,E}^{j\gr{e_{d+1}}} U_{d+1}^{s\pr{E}} P_{d+1,E}^{\gr{e_{d+1}}}f
  =P_{d+1,E}^{j\gr{e_{d+1}}} f
 \end{equation}
 where $s\pr{E}=1$ if $d+1 \in E$ and $-1$ otherwise, 
 \end{enumerate}

 \end{Proposition}
 
 \begin{proof}
  \begin{enumerate}
   \item It follows from the fact that for any function $h$ and 
   any sub-$\sigma$-algebra $\mathcal G$ of $\f$, $\norm{h-\E{h\mid \mathcal G}}
   \leq \norm{h}$ combined with an induction argument.
   \item That $P_E^{\gr{k}}f$ belongs 
 to $\h_E$  follows from the definition of $P_E$. Let $E\subsetneq 
 [d]$. By \eqref{eq:definition_de_Pe} and complete commutativity of 
 $\pr{T^{-\gr{i}}\f_{\gr{0}}}_{\gr{i}\in \Z^d}$, we have 
 \begin{equation}\label{eq:Intermediate_step_proof_item2}
   P_E^{\gr{j}}\circ P_E^{\gr{k}}f=
   \sum_{I,J\subset E}\pr{-1}^{\abs{I}+\abs{J}}
   \E{U^{\pr{\gr{k}+\gr{j}}\cdot\eps\pr{E}} f
   \mid \f_{\min\ens{ \gr{j}\cdot\eps\pr{E}+\infty\gr{1_I}, 
   \infty\gr{1_J}}}}.
 \end{equation}
 Moreover, since $I$ is contained in $E$, we have, by definition of 
 $\eps\pr{E}$ that 
 \begin{equation}
  \min\ens{ \gr{j}\cdot\eps\pr{E}+\infty\gr{1_I}, 
   \infty\gr{1_J}}=
   \min\ens{ -\gr{j_{E\setminus I}} +\infty\gr{1_I}, 
   \infty\gr{1_J}}.
 \end{equation}
If $E\setminus I$ contains some $i_0$, then 
\begin{equation*}
  \sum_{J\subset E}\pr{-1}^{ \abs{J}}
   \E{U^{\pr{\gr{k}+\gr{j}}\cdot\eps\pr{E}} f
   \mid \f_{\min\ens{ \gr{j}\cdot\eps\pr{E}+\infty\gr{1_I}, 
   \infty\gr{1_J}}}} =\sum_{J\subset E\setminus\ens{i_0}}
   \pr{-1}^{ \abs{J}}a_J,
\end{equation*}
where 
\begin{equation*}
a_J:=  \E{ U^{\pr{\gr{k}+\gr{j}}\cdot\eps\pr{E}} f
   \mid \f_{  \min\ens{\infty\gr{1_I}-\gr{j_{E\setminus I}}  ;
   \infty\gr{1_J}
   }}}
   - \E{ U^{\pr{\gr{k}+\gr{j}}\cdot\eps\pr{E}} f
   \mid \f_{  \min\ens{\infty\gr{1_I}-\gr{j_{E\setminus I}}  ;
   \infty\gr{1_{J\cup\ens{i_0}}}
   }}}
\end{equation*}

and the latter term equals $0$. Consequently, in \eqref{eq:Intermediate_step_proof_item2}, 
only the term where $I=E$ appears, which gives \eqref{eq:semi_group}. 
Assume that $E=[d]$. Define for $\gr{i}\in \N^d$,
$Q_1^{\gr{i}}f:=U^{-\gr{i}}\pr{f-\E{f\mid \f_{\infty\gr{1_{[d]}}}}}$
and $Q_2^{\gr{i}}f:=\sum_{J\subset [d]}\pr{-1}^{\abs{J}}
\E{U^{-\gr{j}}f\mid \f_{\infty\gr{1_J}}}$. Then $P_{[d]}^{\gr{i}}f
=Q_1^{\gr{i}}+Q_2^{\gr{i}}$, $Q_1^{\gr{j}+\gr{k}}f=Q_1^{\gr{j}}
Q_1^{\gr{k}}\pr{f}$ and by similar arguments as in the case where 
$E\subsetneq [d]$, $Q_2^{\gr{j}+\gr{k}}f=Q_2^{\gr{j}}
Q_2^{\gr{k}}\pr{f}$. Therefore, establishing \eqref{eq:semi_group} 
reduces $Q_1^{\gr{j}}Q_2^{\gr{k}}f=
Q_2^{\gr{j}}Q_1^{\gr{k}}f=0$ which holds because $\E{U^{\gr{a}}
Q_1^{\gr{b}}f\mid \f_{\infty\gr{1_{[d]}}}}=0$ for any $\gr{a}\in \Z^d$ 
and $\gr{b}\in\N^d$.

   \item When $d+1\notin E$, this follows from item~\ref{itm:semi_group}. 
   Assume now that $d+1\in E$. For any function $h$, the following 
   equality holds:
   \begin{equation}
    \til{P^{\gr{j}}_{d,E\setminus \ens{d+1}}}
    =P_{d+1,E}^{\gr{j}}h
    +\sum_{J\subset E\setminus \ens{d+1}}
    \pr{-1}^{\abs{J}+\abs{E}}
    \E{U^{\gr{j}\cdot \eps\pr{E}}h\mid \f_{\infty\gr{1_J}}},
   \end{equation}
  hence, by item~\ref{itm:semi_group}, it suffices to establish that 
  \begin{equation}
   \sum_{J\subset E\setminus \ens{d+1}}
    \pr{-1}^{\abs{J}+\abs{E}}
    \E{U^{\gr{j}\cdot \eps\pr{E}}P_{d+1,E}^{\gr{e_{d+1}}}f
    \mid \f_{\infty\gr{1_J}}}=0,
  \end{equation}
 which follows from the fact that $U^{\gr{j}\cdot \eps\pr{E}}P_{d+1,E}^{\gr{e_{d+1}}}f$ 
 belongs to $\h_{d+1,E}$.
 \item Noticing that $P_{d+1,E}^{\gr{k}}f=
 P_{d+1,E}^{\gr{0}}\pr{U^{\gr{k}\cdot\eps{E}}f}$, we derive that
 \begin{equation}
  P_{d+1,E}^{j\gr{e_{d+1}}} U_{d+1}^{s\pr{E}}f=
  P_{d+1,E}^{\gr{0}}\pr{U^{j\gr{e_{d+1}}\cdot \eps\pr{E}   } U_{d+1}^{s\pr{E}}
  f}
  =P_{d+1,E}^{\gr{0}}\pr{U^{\pr{j+1}\gr{e_{d+1}}\cdot \eps\pr{E}   }  
  f},
 \end{equation}
 which entails the wanted result by an application of item~\ref{itm:semi_group}.
  \end{enumerate}

 \end{proof}

\subsubsection{Intermediate steps}

The proof of Theorem~\ref{thm:OMC_decomposition} will
require the following lemmas.
  
  \begin{Lemma}\label{lem:decomposition_gE}
  Let $E\subsetneq [d]$ and let $h$ be an integrable 
  $\f_{\infty\gr{1_I}}$-measurable function such that $\E{h\mid\f_{\infty\gr{J}}}
  =0$ for each $J\subsetneq I$. Then the function 
  $g_E:=\prod_{q=1}^d\pr{\Id-P_E^{\gr{e_q}}}h$ admits the decomposition 
  \begin{equation}\label{eq:OMC_decomposition_PE}
  g_E:= m_k^E+\sum_{\substack{J\subset [d]\\ J\neq \emptyset }}
    \prod_{j\in J}\pr{\Id-U_j}m_{k,J}^E,
   \end{equation}
  where  $\pr{m_k^E\circ T^{\mathbf i}}_{\mathbf i\in \Z^d}$ is an 
  orthomartingale differences random field with respect to the filtration 
  $\pr{T^{-\mathbf i}\f_{\gr{0}}}_{\mathbf i\in \Z^d}$.
  and for each nonempty subset $J$ of $[d]$ such that $J\neq [d]$, the random field 
  $\pr{m_{k,J}^E\circ T^{\gr{i_J} }}_{\mathbf i\in 
  \Z^d}$  is an
  orthomartingale differences random field with respect to the filtration 
  $\pr{T^{-\gr{i_J}}\f_{\infty\gr{1_{J^c}}    } }_{\mathbf i\in \Z^d  }$.
  \end{Lemma}
  
  \begin{Lemma}\label{lem:decomposition_gd}
   Let $h$ be an integrable function. Then the function 
   $g_{[d]}:=\prod_{q=1}^d\pr{\Id-P_{[d]}^{\gr{e_q}}}h$ admits the decomposition 
  \begin{equation}\label{eq:OMC_decomposition_Pd}
  g_{[d]}:= m_k^{[d]}+\sum_{\substack{J\subset [d]\\ J\neq \emptyset }}
    \prod_{j\in J}\pr{\Id-U_j}m_{k,J}^{[d]},
   \end{equation}
  where  $\pr{m_k^{[d]}\circ T^{\mathbf i}}_{\mathbf i\in \Z^d}$ is an 
  orthomartingale differences random field with respect to the filtration 
  $\pr{T^{-\mathbf i}\f_{\gr{0}}}_{\mathbf i\in \Z^d}$.
  and for each nonempty subset $J$ of $[d]$ such that $J\neq [d]$, the random field 
  $\pr{m_{k,J}^{[d]}\circ T^{\gr{i_J} }}_{\mathbf i\in 
  \Z^d}$  is an
  orthomartingale differences random field with respect to the filtration 
  $\pr{T^{-\gr{i_J}}\f_{\infty\gr{1_{J^c}}    } }_{\mathbf i\in \Z^d  }$.
  \end{Lemma}
  
  \begin{Lemma}\label{lem:simplified_decomposition}
   For any square integrable function $H$ and any $K\subset [d]$, 
   the function 
   \begin{equation}
    F:=\E{\prod_{q\in [d]\setminus K}\pr{\Id-U_q}H\mid 
   \f_{\infty\gr{1_K}}}
   \end{equation}
    admits the decomposition
   \begin{equation}
    F=\sum_{K'\subset [d]\setminus K}\prod_{q\in K'}\pr{\Id-U_q}m_{K'},
   \end{equation}
  where $m_{K'}$ is $\f_{\infty\gr{1_{[d]\setminus K'}}}$-measurable 
  and if $S\subsetneq K'$, then $\E{m_{K'}\mid \f_{\infty\gr{1_S}}}=0$.
  \end{Lemma}
  \begin{proof}
   Define for $\gr{i}\in \N^d$,
   \begin{equation}
    Q^{\gr{i}}f:=\E{U^{\gr{i}}f\mid  \f_{\infty\gr{1_K}}}
   \end{equation}
  and $Q_q^if:=Q^{i\gr{e_q}}f$. Then $F=\prod_{q\in [d]\setminus K}\pr{\Id-
  Q_q}H$. Moreover, since $\Id-Q_q=\Id-U_q^{-1}Q_q+U_q^{-1}Q_q-Q_q$ 
  and $Q_{q'}$ commutes with $U_q^{-1}Q_q$ for $q\neq q'$, we derive that 
  \begin{equation}
   F=\sum_{K'\subset [d]\setminus K}
   \prod_{q\in K'}
   \pr{U_q^{-1}Q_q-Q_q}\prod_{q'\in ([d]\setminus K)\setminus K'}
   \pr{\Id-U_{q'}^{-1}Q_{q'}}H.
  \end{equation}
 If $q\neq q'$, then $U_q^{-1}$ commutes with $Q_{q'}$. 
 Therefore, the following equalities hold
 \begin{align}
   F&=\sum_{K'\subset [d]\setminus K} 
   \prod_{q\in K'}
   \pr{U_q^{-1} -\Id}\prod_{q\in K'}Q_q
   \prod_{q'\in ([d]\setminus K)\setminus K'}\pr{\Id-U_{q'}^{-1}Q_{q'}}H\\
   &=\sum_{K'\subset [d]\setminus K} 
   \prod_{q\in K'}
   \pr{\Id-U_q}\prod_{q\in K'}U_q^{-1}Q_q
   \prod_{q'\in ([d]\setminus K)\setminus K'}\pr{\Id-U_{q'}^{-1}Q_{q'}}H,
  \end{align}
  which gives the wanted decomposition.
  \end{proof}

  \begin{proof}[Proof of Lemma~\ref{lem:decomposition_gE}]
   Observe that 
   \begin{align}
    g_E&=\sum_{F\subset [d]} \pr{-1}^{\abs{F}}P_E^{\gr{1_F}}h\\
    &=\sum_{F\subset [d]} \pr{-1}^{\abs{F}}\sum_{J\subset E}
    \pr{-1}^{\abs{J}}\E{ 
    U^{\gr{1_F}\cdot \eps\pr{E}} h\mid \f_{\infty\gr{1_J}}
    }\\
    &=\sum_{J\subset E}
    \pr{-1}^{\abs{J}}\E{ \prod_{q=1}^d\pr{\Id-U_q^{\eps\pr{E}_q}}  
    h\mid \f_{\infty\gr{1_J}}
    }\\
    &=\sum_{J\subset E}
    \pr{-1}^{\abs{J}}\pr{-1}^{\abs{J\cap E}}\prod_{j\in J}\pr{\Id-U_j}
    \E{ \prod_{q\in [d]\setminus J }\pr{\Id-U_q^{\eps\pr{E}_q}}  
    \prod_{l \in J}U_lh\mid \f_{\infty\gr{1_J}}
    }.
   \end{align}
  We then apply Lemma~\ref{lem:simplified_decomposition} to each $J\subset E$ 
  with $H$ such that $\prod_{q\in [d]\setminus J }\pr{\Id-U_q^{\eps\pr{E}_q}}  
    \prod_{l \in J}U_l h=\prod_{q\in [d]\setminus J}
    \pr{\Id-U_q}H$.
  \end{proof}
 \begin{proof}[Proof of Lemma~\ref{lem:decomposition_gd}]
  We start from the following inequalities:
  \begin{align}
   g_{[d]}&=\sum_{F\subset [d]} \pr{-1}^{\abs{F}}P_E^{\gr{1_F}}h\\
    &=\sum_{F\subset [d]} \pr{-1}^{\abs{F}}\sum_{J\subsetneq E}
    \pr{-1}^{\abs{J}}\E{ 
    U^{\gr{1_F}\cdot \eps\pr{E}} h\mid \f_{\infty\gr{1_J}}
    }+\sum_{F\subset [d]} \pr{-1}^{\abs{F}}U^{-\gr{1_F}}h\\
    &=\sum_{J\subsetneq  E}
    \pr{-1}^{\abs{J}}\E{ \prod_{q=1}^d\pr{\Id-U_q^{\eps\pr{E}_q}}  
    h\mid \f_{\infty\gr{1_J}}
    }+\prod_{q=1}^d\pr{\Id-U_q^{-1}}h 
  \end{align}
 and we apply Lemma~\ref{lem:simplified_decomposition} to 
 each $J\subsetneq E$ and a $H$ such that $\prod_{q\in [d]\setminus J }\pr{\Id-U_q^{\eps\pr{E}_q}}  
    \prod_{l \in J}U_l^h=\prod_{q\in [d]\setminus J}
    \pr{\Id-U_q}H$.
 \end{proof}

 \subsection{Proof of Theorem~\ref{thm:CNS_orthomartingale_approximation}}

   \begin{Lemma}\label{lem:bounded_sums}
   Let $A_1,\dots, A_d$ be commuting operators from a closed 
   subspace $V$ of $\mathbb L^2$ to itself, and $A^{\gr{i}}:=
   A_1^{i_1}\dots A_d^{i_d}$. Let $F$ be a 
   function such that $\sup_{\gr{n}\in \N^d}
   \norm{\sum_{\gr{0}\imd \gr{i} \imd \gr{n}  }  A^{\gr{i} } F}<+\infty$. Then 
   there exists a function $h\in V$ such that 
   $F=\prod_{q=1}^d\pr{\Id-A_q}h$.
  \end{Lemma}
  \begin{proof}
   We use the idea of proof of Lemma~5 in \cite{MR0096975}. We define 
   \begin{equation}
    f_n:= \frac 1{n^d}\sum_{\gr{1}\imd\gr{k}\imd \gr{n}}
     \sum_{\gr{0}\imd \gr{i} \imd
    \gr{k}-\gr{1}}\prod_{q=1}^d\pr{\Id-A_q}A^{\gr{i}}F
   \end{equation}
  Using the assumption on $F$, we derive that 
  \begin{equation}
   \lim_{n\to +\infty}\norm{F-f_n}=0.
  \end{equation}
 Moreover, defining 
 \begin{equation}
  h_n:=\frac 1{n^d}\sum_{\gr{1}\imd\gr{k}\imd \gr{n}}
     \sum_{\gr{0}\imd \gr{i} \imd
    \gr{k}-\gr{1}}A^{\gr{i}} F
 \end{equation}
we observe that the sequence $\pr{\norm{h_n}}_{n\geq 1}$ is bounded. 
Since $\mathbb L^2$ is reflexive, there exists a subsequence
$\pr{h_{n_k}}_{k\geq 1}$ which converges weakly in $\mathbb L^2$ 
to some $h$. Then the sequence $\pr{f_{n_k}}_{k\geq 1}=
\pr{\prod_{q=1}^d\pr{\Id-A_q}h_{n_k}}_{k\geq 1}$ 
converges weakly to $\prod_{q=1}^d\pr{\Id-A_q}h$. By uniqueness of the limit, 
 equality 
$F=\prod_{q=1}^d\pr{\Id-A_q}h$ holds. That $h$ belongs to $V$ follows from 
closedness of $V$.
  \end{proof}

\begin{proof}[Proof of Theorem~\ref{thm:OMC_decomposition}]
We prove sufficiency (necessity can be checked by direct computations). We
use the idea of proof of Proposition~4.1 in 
\cite{MR3178617}. Since $f= \sum_{E\subset  [d]}
 P_Ef$, it suffices to find an orthomartingale-coboundary 
decomposition for $P_Ef$ for any subset $E$ of $[d]$. 
To this aim, we apply Lemma~\ref{lem:bounded_sums} to the following setting: $V=\h_E$, 
$F=P_Ef$ and $A_q^i:=P_E^{i\gr{e_q}}$. We then conclude by 
Lemmas~\ref{lem:decomposition_gE} and \ref{lem:decomposition_gd}.
 
\end{proof}

  \subsubsection{Construction of the approximating martingale}

  The combination of Lemma~\ref{lem:decomposition_gE} and 
  \ref{lem:decomposition_gd}  shows that $B_k(f)$ admits an orthomartingale-coboundary 
  decomposition.

  \begin{Lemma}\label{lem:OMC_decomposition_of_Bk}
   For each integer $k\geqslant 1$ and 
   each integrable and measurable function $f\colon \Omega\to \R$, the 
   function $B_k(f)$ can be written in the following way:
   \begin{equation}\label{eq:OMC_decomposition}
    B_k(f)=m_k+\sum_{\substack{J\subset [d]\\ J\neq \emptyset }}
    \prod_{j\in J}\pr{\Id-U_j}m_{k,J},
   \end{equation}
  where 
   $\pr{m_k\circ T^{\mathbf i}}_{\mathbf i\in \Z^d}$ is an 
  orthomartingale differences random field with respect to the filtration 
  $\pr{T^{-\mathbf i}\f_{\gr{0}}}_{\mathbf i\in \Z^d}$.
  and for each nonempty subset $J$ of $[d]$ such that $J\neq [d]$, the random field 
  $\pr{m_{k,J}\circ T^{\gr{i_J} }}_{\mathbf i\in 
  \Z^d}$  is an
  orthomartingale differences random field with respect to the filtration 
  $\pr{T^{-\gr{i_J}}\f_{\infty\gr{1_{J^c}}    } }_{\mathbf i\in \Z^d  }$.
  \end{Lemma}

  Considering the notations of Lemma~\ref{lem:OMC_decomposition_of_Bk}, we 
  introduce the following notation:
  \begin{equation}
   C_k\pr{f}:=\sum_{\substack{J\subset [d]\\ J\neq \emptyset }}
    \prod_{j\in J}\pr{\Id-U_j}m_{k,J}, \quad k\geq 1.
  \end{equation}
  Therefore, for any $k\geq 1$, the following equality holds 
  \begin{equation}\label{eq:decomposition_of_f}
   f=f-B_k\pr{f}+m_k+C_k\pr{f}.
  \end{equation}

  \subsubsection{Sufficiency} 
   
   Lemma~\ref{lem:OMC_decomposition_of_Bk} gave a sequence of functions 
   $\pr{m_k}_{k\geqslant 1}$ such that $\pr{m_k\circ T^{\mathbf i}}_{\mathbf i\in \Z^d}$ is an 
  orthomartingale differences random field with respect to the filtration 
  $\pr{T^{-\mathbf i}\f_{\gr{0}}}_{\mathbf i\in \Z^d}$. 
  Now, we have that to show that 
  if \eqref{eq:sufficient_cond_MA} holds, then 
  the sequence $\pr{m_k}_{k\geqslant 1}$
  is convergent, and that the limiting function $m$
  satisfies \eqref{eq:def_orthomartingale_approximation}.

   \begin{Lemma}\label{lem:neglig_cobords}
   For any function $f$, and any $k\geq 1$, $\norm{C_k\pr{f}}_+=0$.
   \end{Lemma}
   \begin{proof}
    It suffices to prove that for any non-empty subset $J$ of 
    $[d]$ and any $k\geq 1$, 
    \begin{equation}\label{eq:goal_lemma_neglig_cobord}
      \norm{\prod_{j\in J}\pr{\Id-U_j}m_{k,J}}_+=0.
    \end{equation}
   Observe that 
  \begin{align}
  \max_{\mathbf 1\preccurlyeq \mathbf i\preccurlyeq\mathbf n}
  \abs{S_{\mathbf i}\pr{\prod_{j\in J}\pr{\Id-U_j}m_{k,J}  }}
  &=\max_{\mathbf 1\preccurlyeq \mathbf i\preccurlyeq\mathbf n}\abs{
  \prod_{j\in J}\pr{\Id-U_j^{i_j} }S_{\gr{i_{J^c}}}
   \pr{m_{k,J}}}\\
  &\leqslant 2^{\abs{J}}
  \max_{\mathbf 0\preccurlyeq \mathbf i\preccurlyeq\mathbf n}\abs{
   U^{\gr{i_J}} S_{\gr{i_{J^c}+\gr{1_J}}}\pr{m_{k,J}}}.
  \end{align}
 Since $J$ is not empty, it contains some $q$. Using the fact that 
 $\norm{\max_{i\in I}\abs{Y_i}}\leq \abs{I}^{1/2}
 \max_{i\in I}\norm{\abs{Y_i}}$, we derive that for any $\gr{n}\smd \gr{1}$, 
 \begin{equation}
  \frac 1{\abs{\gr{n}}} 
  \E{\max_{\mathbf 1\preccurlyeq \mathbf i\preccurlyeq\mathbf n}
  \abs{S_{\mathbf i}\pr{\prod_{j\in J}\pr{\Id-U_j}m_{k,J}  }}^2  }
  \leq \frac{2^{2\abs{J}}}{n_q}\E{\max_{0\leq l\leq n_q}
  U_q^{l} Y_{\gr{n}}},
 \end{equation}
 where 
 \begin{equation}
  Y_{\gr{n}}:= \frac 1{\abs{\gr{n_{[d]\setminus J}}}}\max_{\gr{1_{[d]\setminus J}\imd\gr{i}\imd\gr{n_{[d]\setminus J}}}}
  \abs{S_{\gr{i_{J^c}+\gr{1_J}}}\pr{m_{k,J}}}^2.
 \end{equation}
 Then for any $R>0$, 
 \begin{equation}
  \frac{2^{2\abs{J}}}{n_q}\E{\max_{0\leq l\leq n_q}
  U_q^{l} Y_{\gr{n}}} \leq \frac{2^{2\abs{J}}}{n_q}R+
  2^{2\abs{J}} \E{Y_{\gr{n}}\mathbf 1\ens{Y_{\gr{n}}>R} }.
 \end{equation}
 By 
 Proposition~\ref{prop:uniform_integrability}, the family 
 $\ens{Y_{\gr{n}},\gr{n}\smd \gr{1}}$ is uniformly integrable. 
 This gives \eqref{eq:goal_lemma_neglig_cobord} and
 ends the proof of Lemma~\ref{lem:neglig_cobords}.
   \end{proof}

  \begin{Lemma}\label{lem:construction_approximation_martingale}
   Let $f\colon\Omega\to \R$ be a measurable square integrable function such that 
   \eqref{eq:sufficient_cond_MA} holds. Then the 
   sequence $\pr{m_k}_{k\geqslant 1}$ is convergent in $\mathbb L^2$ to some function $m$. 
  \end{Lemma}

  \begin{proof}
   Let $k$ and $l$ be fixed positive integers. Since 
   $\pr{\pr{m_k-m_l}\circ T^{\mathbf i}}_{\mathbf i\in \Z^d}$ is an 
   orthomartingale differences random field with respect to the filtration 
   $\pr{T^{-\mathbf i}\f_{\gr{0}}}_{\mathbf i\in \Z^d}$, we have for each 
   positive integer $n$, by orthogonality of 
   increments, 
   \begin{equation}\label{eq:norme_de_mk-ml}
    \norm{m_k-m_l}=\frac 1{n^{d/2}}\norm{S_{n\mathbf 1}\pr{m_k-m_l}}.
   \end{equation}
  By Lemma~\ref{lem:OMC_decomposition_of_Bk} and 
  \eqref{eq:decomposition_of_f}, the following equality holds
  \begin{equation}
   S_{n\mathbf 1}\pr{m_k-m_l}=S_{n\mathbf 1}\pr{B_k\pr{f}-f}-
   S_{n\mathbf 1}\pr{B_l\pr{f}-f}-S_{n\mathbf 1}\pr{C_k\pr{f}}
   +S_{n\mathbf 1}\pr{C_l\pr{f}}
  \end{equation}
  hence taking the $\mathbb L^2$ norm, we get 
  \begin{equation*}
   \norm{S_{n\mathbf 1}\pr{m_k-m_l}}
   \leqslant \norm{S_{n\mathbf 1}\pr{B_k\pr{f}-f}}+\norm{S_{n\mathbf 1}\pr{B_l\pr{f}-f}} 
   +\norm{S_{n\mathbf 1}\pr{C_k\pr{f}}}+
   \norm{S_{n\mathbf 1}\pr{C_l\pr{f}}}.
  \end{equation*}
  Dividing on both sides by $n^{d/2}$ and letting $n$ going to infinity, 
  we get by Lemma~\ref{lem:neglig_cobords} and \eqref{eq:norme_de_mk-ml}
   \begin{equation}
   \norm{m_k-m_l}\leqslant \norm{B_k\pr f-f}_++\norm{B_l\pr f-f}_+.
   \end{equation}
  This proves that the sequence $\pr{m_k}_{k\geqslant 1}$ is Cauchy in $\mathbb L^2$ 
  hence convergent to some function $m$. This ends the proof of 
  Lemma~\ref{lem:construction_approximation_martingale}.
  \end{proof}
 
 Since for each $k$, the function $m_k$ is $\f_{\gr{0}}$-measurable, the 
 function $m$ is $\f_{\gr{0}}$-measurable. Moreover, we have for each $q\in 
 [d]$, $\mathbb E\left[m_k\mid T_q\f_{\gr{0}}\right]=0$ hence 
 $\mathbb E\left[m\mid T_q\f_{\gr{0}}\right]=0$, which proves that 
 $\pr{m\circ T^{\mathbf i}}_{\mathbf i\in \Z^d}$ is an orthomartingale 
 differences random field. 
 
 The purpose of the following lemma is the verification that $m$ gives the 
 wanted approximation.
 
 \begin{Lemma}\label{lem:orthomartingale_m_works}
  Let $f\colon\Omega\to \R$ be a measurable square integrable function such that 
   \eqref{eq:sufficient_cond_MA} holds and let $m$ be the 
   function given by Lemma~\ref{lem:construction_approximation_martingale}. Then 
   \eqref{eq:def_orthomartingale_approximation} takes place.
 \end{Lemma}
  
 \begin{proof}[Proof of Lemma~\ref{lem:orthomartingale_m_works}]
  Let $k\geqslant 1$ be an arbitrary but fixed integer. For any $\mathbf n\in \N^d$ 
  such that $\mathbf n \succcurlyeq \mathbf 1$, we have, using 
  Proposition~\ref{prop:Doob} with $M:=m-m_k$, 
  \begin{equation}
   \frac 1{\abs{\mathbf n}^{1/2}}
   \norm{\max_{\mathbf 1\preccurlyeq \mathbf i\preccurlyeq \mathbf n}\abs{S_{\mathbf i}
   \pr{f-m}}}\leqslant \frac 1{\abs{\mathbf n}^{1/2}}
   \norm{\max_{\mathbf 1\preccurlyeq \mathbf i\preccurlyeq \mathbf n}\abs{S_{\mathbf i}
   \pr{f-m_k}}}+2^d\norm{m-m_k}.
  \end{equation}
 Now, we use the inequality 
  \begin{equation*}
    \frac 1{\abs{\mathbf n}^{1/2}}
   \norm{\max_{\mathbf 1\preccurlyeq \mathbf i\preccurlyeq \mathbf n}\abs{S_{\mathbf i}
   \pr{f-m_k}}}\leqslant 
    \frac 1{\abs{\mathbf n}^{1/2}}
   \norm{\max_{\mathbf 1\preccurlyeq \mathbf i\preccurlyeq \mathbf n}\abs{S_{\mathbf i}
   \pr{f-B_k(f)}}}+
    \frac 1{\abs{\mathbf n}^{1/2}}
   \norm{\max_{\mathbf 1\preccurlyeq \mathbf i\preccurlyeq \mathbf n}\abs{S_{\mathbf i}
   \pr{m_k-B_k(f)}}}
  \end{equation*}
 and take the $\limsup$ as $\mathbf n$ goes to infinity to obtain that for any 
 $k\geqslant 1$,
 \begin{equation}\label{eq:controle_maxima_f-m}
  \limsup_{\mathbf n\to \infty}
   \frac 1{\abs{\mathbf n}^{1/2}}
   \norm{\max_{\mathbf 1\preccurlyeq \mathbf i\preccurlyeq \mathbf n}\abs{S_{\mathbf i}
   \pr{f-m}}}\leqslant \norm{B_k(f)-f}_++\norm{C_k(f)}_+
   +2^d\norm{m-m_k}. 
 \end{equation}
 By \eqref{eq:sufficient_cond_MA}, 
 Lemmas~\ref{lem:neglig_cobords} and \ref{lem:construction_approximation_martingale}, we 
 get that 
 the right hand side of \eqref{eq:controle_maxima_f-m} converges to $0$ as $k$ 
 goes to infinity. This concludes the proof of Lemma~\ref{lem:orthomartingale_m_works} and 
 that of Theorem~\ref{thm:CNS_orthomartingale_approximation}.
 \end{proof}

\subsection{Proof under projective conditions}

  \subsubsection{Hannan's condition}
  
Lemma~5.2 in \cite{MR3264437} states the following inequality: for 
any function $f$ satisfying the conditions of Proposition~\ref{prop:Hannan} 
and any $\gr{n}\in \pr{\N^*}^d$, 
\begin{equation}\label{eq:inegalite_maximale_Hannan}
 \norm{\max_{\gr{1}\imd\gr{i}\imd\gr{n}}
 \abs{S_{\gr{i}}\pr{f}}} \leq 2^d\abs{\gr{n}}^{1/2}
 \sum_{\gr{i}\in\Z^d}\norm{\pi_{\gr{i}}\pr{f}}.
\end{equation}
We shall check \eqref{eq:easier_sufficient_cond}. To this aim, 
we fix a nonempty subset $J$ of $[d]$ and $E\subset [d]$ and we 
apply \eqref{eq:inegalite_maximale_Hannan} to the function 
$k^{-\abs{J}}\sum_{\gr{1_J}\imd \gr{j}\imd k\gr{1_J}}P_E^{\gr{j}}f$ (which 
satisfies the assumptions of Proposition~\ref{prop:Hannan} because so does $f$)
in order to obtain
\begin{equation}
 \norm{k^{-\abs{J}}\sum_{\gr{1_J}\imd \gr{j}\imd k\gr{1_J}}P_E^{\gr{j}}f}_+
 \leq 2^d\sum_{\gr{i}\in\Z^d}\norm{\pi_{\gr{i}}
 \pr{k^{-\abs{J}}\sum_{\gr{1_J}\imd \gr{j}\imd k\gr{1_J}}P_E^{\gr{j}}f}}.
\end{equation}
Define $I_E:=\ens{\gr{i}\in \Z^d\mid i_q> 0\mbox{ if and only if }q\in E}$.
Then for $\gr{i}\in \Z^d\setminus I_E$, $
\pi_{\gr{i}}
 \pr{k^{-\abs{J}}\sum_{\gr{1_J}\imd \gr{j}\imd k\gr{1_J}}P_E^{\gr{j}}f}=0$
and if $\gr{i}$ belongs to $I_E$, then 
\begin{equation}
\norm{\pi_{\gr{i}}
 \pr{k^{-\abs{J}}\sum_{\gr{1_J}\imd \gr{j}\imd k\gr{1_J}}P_E^{\gr{j}}f}}
 \leq k^{-\abs{J}}\sum_{\gr{1_J}\imd \gr{j}\imd k\gr{1_J}}
 \norm{\pi_{\gr{i}-\gr{j}\cdot \eps\pr{E} } }
\end{equation}
hence 
\begin{equation}
 \sum_{\gr{i}\in I_E}\norm{\pi_{\gr{i}}
 \pr{k^{-\abs{J}}\sum_{\gr{1_J}\imd \gr{j}\imd k\gr{1_J}}P_E^{\gr{j}}f}}
 \leq k^{-\abs{J}}\sum_{\gr{1_J}\imd \gr{j}\imd k\gr{1_J}}
 \sum_{\gr{i}\in I_E}
 \norm{\pi_{\gr{i}-\gr{j}\cdot \eps\pr{E} } }.
\end{equation}
Since $J$ is nonempty, we can choose $q\in J$. Observe that 
\begin{equation}
 \sum_{\gr{1_J}\imd \gr{j}\imd k\gr{1_J}}
 \sum_{\gr{i}\in I_E}
 \norm{\pi_{\gr{i}-\gr{j}\cdot \eps\pr{E} } } \leq 
 k^{\abs{J}-1}\sum_{j=1}^k\sum_{\gr{i}\in \Z^d: \abs{i_q}>j }
 \norm{\pi_{\gr{i} }},
\end{equation}
hence 
\begin{equation}
 \norm{k^{-\abs{J}}\sum_{\gr{1_J}\imd \gr{j}\imd k\gr{1_J}}P_E^{\gr{j}}f}_+ 
 \leq 2^d\frac 1k\sum_{j=1}^k\sum_{\gr{i}\in \Z^d: \abs{i_q}>j }
 \norm{\pi_{\gr{i} }}.
\end{equation}
That \eqref{eq:easier_sufficient_cond} is satisfied follows from 
finiteness of $\sum_{\gr{i}\in \Z^d }
 \norm{\pi_{\gr{i} }}$. This concludes the proof of Theorem~\ref{prop:Hannan}.

  \subsubsection{Maxwell and Woodroofe condition}
  
\begin{proof}[Proof of Proposition~\ref{prop:inegalite_maximale}]

As in \cite{MR2123210,MR2255301,2014arXiv1403.0772C}, the proof will 
be done by dyadic induction.

 We shall prove by induction on $d$ the following assertion: there 
 exists constants $K(d)$, $C(d,J)$, $\emptyset\subsetneq J\subset [d]$ 
 such that for any commuting invertible measure preserving maps $T_1,\dots,T_d$, 
 any sub-$\sigma$-algebra $\f_{\gr{0}}$ of $\f$ such that 
 $\pr{T^{-\gr{i}}\f_{\gr{0}}}_{\gr{i}\in\Z^d}$ is a commuting filtration, any 
 subset $E$ of $[d]$
 any  function $f\in 
 \mathcal H_{d,E}$ and any $\gr{n}\in \N^d$, 
 \begin{multline}\label{eq:maximal_inequality_dim_d}
  \norm{\max_{\gr{1}\imd \gr{i}\imd \gr{2^n}}  \abs{S_{\gr{i}}\pr{f}} } \leq 
  \abs{\gr{2^n}}^{1/2}K\pr{d}\norm{f}\\ + \abs{\gr{2^n}}^{1/2}
  \sum_{\emptyset\subsetneq J\subset [d]} 
  C\pr{d,J}\sum_{\gr{0_J}\imd \gr{i}\imd \gr{n_J}}
  \abs{\gr{2^{i_J}}}^{-1/2}  \norm{ \sum_{\gr{1_J} \imd \gr{j}\imd \gr{\pr{2^i}_J}    } 
  P_E^{\gr{j}} f}
 \end{multline}
 where $P_E$ is defined by \eqref{eq:definition_de_Pe}.

 The constants are defined recursively in the following way:
 \begin{equation}\label{eq:recurrence_Kd}
  K\pr{d+1}=6K\pr{d},
 \end{equation}
 if $J$ is a nonempty subset of $[d]$, then 
 \begin{equation}\label{eq:recurrence_Cd+1_J}
  C\pr{d+1,J}=4C\pr{d,J}
 \end{equation}
 \begin{equation}\label{eq:recurrence_Cd+1_J_bis}
  C\pr{d+1,J\cup\ens{d+1}}=2\sqrt 2C\pr{d+1,J}+C\pr{d,J}.
 \end{equation}
and 
\begin{equation}\label{eq:recurrence_singleton_d+1}
 C\pr{d+1,\ens{d+1}}=\sqrt 2 K\pr{d+1}.
\end{equation}

When $d=1$, the result was established in Proposition~2.3. of \cite{MR2123210} when 
the function $f$ is $\f_0$-measurable and was extended to the nonadapted 
case in Proposition~1 of \cite{MR2344817}.

Now, assume that the result holds for some $d\geq 1$ and let us prove it for 
$d+1$. This will be done by induction on $n_{d+1}$. More precisely, we consider the 
following assertion $\mathcal P(m)$ defined as "there 
 exists constants $K(d+1)$, $C(d+1,J)$, $\emptyset\subsetneq J\subset [d+1]$ 
 such that for any commuting invertible measure preserving maps $T_1,\dots,T_{d+1}$, 
 any sub-$\sigma$-algebra $\f_{\gr{0}}$ of $\f$ such that 
 $\pr{T^{-\gr{i}}\f_{\gr{0}}}_{\gr{i}\in\Z^{d+1}}$ is a completely commuting filtration, 
 any subset $E$ of $[d+1]$,
 any function $f\in \h_{d+1,E}$, any $\gr{n}\in \N^{d+1}$ such that 
 $n_{d+1}\leq m$, 
  \begin{multline}\label{eq:maximal_inequality_dim_d+1}
  \norm{\max_{\gr{1}\imd \gr{i}\imd \gr{2^n}}  \abs{S_{\gr{i}}\pr{f}} } \leqslant 
  \abs{\gr{2^n}}^{1/2}K\pr{d+1}\norm{f}\\ + \abs{\gr{2^n}}^{1/2}
  \sum_{\emptyset\subsetneq J\subset [d+1]}
  C\pr{d+1,J}\sum_{\gr{0_J}\imd \gr{i}\imd \gr{n_J}}
  \abs{\gr{2^i}}^{-1/2}  \norm{ \sum_{\gr{1_J} \imd \gr{j}\imd \gr{\pr{2^i}_J}    } 
  P_E^{\gr{j}} f} ,
 \end{multline}
 where $P_E$ is defined by \eqref{eq:definition_de_Pe}.
 
 The assertion $\mathcal P(0)$ holds by the case of the dimension $d$. 
 Now assume that $\mathcal P(m)$ is true for some $m$ and let us 
 prove $\mathcal P(m+1)$. We thus know that
 \begin{enumerate}[label=\alph*)]
  \item\label{it:dim_d} inequality 
 \eqref{eq:maximal_inequality_dim_d} 
 holds for any commuting invertible measure preserving maps $T_1,\dots,T_d$, 
 any sub-$\sigma$-algebra $\f_{\gr{0}}$ of $\f$ such that 
 $\pr{T^{-\gr{i}}\f_{\gr{0}}}_{\gr{i}\in\Z^d}$ is a commuting filtration, 
 any $E\subset [d]$,
 any function $f\in \h_{d,E}$ and any $\gr{n}\in \N^d$ and
 \item\label{it:HDR_dim_d+1}  for any commuting invertible measure preserving maps $\widetilde{T_1},
 \dots,\widetilde{T_{d+1}}$,  
 any sub-$\sigma$-algebra $\widetilde{\f_{\gr{0}}}$ of $\f$ such that 
 $\pr{\widetilde{T}^{-\gr{i}}\widetilde{\f_{\gr{0}}}}_{\gr{i}\in\Z^{d+1}}$ is 
 a commuting filtration, any $E\subset [d+1]$,
 any function $f\in \h_{d+1,E}$ and
 any $\gr{\widetilde{n}}\in \N^{d+1}$ 
 such that $\widetilde{n_{d+1}}\leq m$, 
 \begin{multline}\label{eq:goal_d+1}
  \norm{\max_{\gr{1}\imd \gr{i}\imd \gr{2^n}}  \abs{S_{\gr{i}}\pr{
  \til{f}}} } \leqslant 
  \abs{\gr{2^n}}^{1/2}K\pr{d+1}\norm{\til{f}}\\ + \abs{\gr{2^n}}^{1/2}
  \sum_{\emptyset\subsetneq J\subset [d+1]} 
  C\pr{d+1,J}\sum_{\gr{0_J}\imd \gr{i}\imd \gr{n_J}}
  \abs{\gr{2^i}}^{-1/2}  \norm{ \sum_{\gr{1_J} \imd \gr{j}\imd \gr{\pr{2^i}_J}    } 
  \til{P_E}^{\gr{j}} \til{f}} ,
 \end{multline}
 where the operators $\til{P_E}$ is defined by \eqref{eq:definition_de_Pe} with 
 $\f_{\infty\gr{1_J}}$ replaced by $\til{\f_{\infty\gr{1_J}}}$.
 \end{enumerate}
Let $T_1,\dots,T_{d+1}$ be commuting invertible measure preserving maps , 
 $\f_{\gr{0}}$ be a sub-$\sigma$-algebra  of $\f$ such that 
 $\pr{T^{-\gr{i}}\f_{\gr{0}}}_{\gr{i}\in\Z^{d+1}}$ is a commuting filtration, 
 $E\subset [d+1]$,
  $f\in \h_{d+1,E}$, and $\gr{n}\in \N^{d+1}$ such that 
  $n_{d+1}\leq m+1$.
 It suffices to prove \eqref{eq:maximal_inequality_dim_d+1} in the case 
 $\gr{n}\in \N^{d+1}$ with $n_{d+1}=m+1$. 
 We define 
 \begin{equation}
  g:= f-U_{d+1}^{s\pr{E}}P_{d+1,E}^{\gr{e_{d+1}}}\pr{f}
 \end{equation}
 where $s\pr{E}=1$ if $d+1 \in E$ and $-1$ otherwise. 
We derive the inequality
 \begin{multline}
  \max_{\gr{1}\imd \gr{i}\imd \gr{2^n}}
  \abs{S_{\gr{i}}\pr{f}} \leq  
   \max_{\gr{1}\imd \gr{i}\imd \gr{2^n}}
  \abs{S_{\gr{i}}\pr{g   }}\\+
   \max_{\gr{1}\imd \gr{i}\imd \gr{2^{n-(m+1)e_{d+1}}}}
  \max_{1\leq j\leq 2^{m}}U_{d+1}^j
  \abs{S_{\gr{i}}\pr{ U_{d+1}^{s\pr{E}}P_{d+1,E}^{\gr{e_{d+1}}}\pr{f}    }
  }
   \\   
  +  \max_{\gr{1}\imd \gr{i}\imd \gr{2^{n-e_{d+1}}}}
  \abs{S_{\gr{i}}\pr{T_1,\dots,T_d,T_{d+1}^2,
  \pr{\Id+U_{d+1}}U_{d+1}^{s\pr{E}}P_E\pr{P_{d+1,E}^{\gr{e_{d+1}}}\pr{f} }  
  }},
 \end{multline}
 where $S_{\gr{i}}\pr{T_1,\dots,T_d,T_{d+1}^2,\cdot}$ is defined 
 like in \eqref{eq:definition_des_sommes_partielles} but $T_{d+1}$ 
 is replaced by $T_{d+1}^2$, hence 
 \begin{equation}\label{eq:decomposition_of_max}
  \norm{\max_{\gr{1}\imd \gr{i}\imd \gr{2^n}}
  \abs{S_{\gr{i}}\pr{f}}} \leq (I)+(II)+(III),
 \end{equation}
 where 
 \begin{equation}
  (I):= \norm{
   \max_{\gr{1}\imd \gr{i}\imd \gr{2^n}}
  \abs{S_{\gr{i}}\pr{g }}}
 \end{equation}
 \begin{equation}
  (II):= \norm{\max_{\gr{1}\imd \gr{i}\imd \gr{2^{n-(m+1)e_{d+1}}}}
  \max_{1\leq j\leq 2^{m}}U_{d+1}^j
  \abs{S_{\gr{i}}\pr{U_{d+1}^{s\pr{E}} P_{d+1,E}^{\gr{e_{d+1}}}\pr{f}     }
  }}\mbox{ and }
 \end{equation}
 \begin{equation}
  (III):=\norm{ \max_{\gr{1}\imd \gr{i}\imd \gr{2^{n-e_{d+1}}}}
  \abs{S_{\gr{i}}\pr{T_1,\dots,T_d,T_{d+1}^2,
  \pr{\Id+U_{d+1}}U_{d+1}^{s\pr{E}}P_{d+1,E}^{\gr{e_{d+1}}}\pr{f} 
  }}}.
 \end{equation}

 If $s(E)=1$, we define the $\sigma$-algebra $\mathcal G_N$ by 
\begin{equation}\label{eq:definition_de_GN}
 \mathcal G_N=\sigma\left(\bigcup_{\gr{i}\in\Z^d}T^{\gr{i}}
T_{d+1}^{-N}\f_{\gr{0}}\right)
\end{equation}
and if $s(E)=-1$, 
\begin{equation} 
 \mathcal G_N=\sigma\left(\bigcup_{\gr{i}\in\Z^d}T^{\gr{i}}
T_{d+1}^{-N-1}\f_{\gr{0}}\right)
\end{equation}

The control of $(I)$ requires the following lemmas. 

 \begin{Lemma}\label{lem:submartingale_property}
  The sequence 
$\pr{ \max_{\gr{1_{[d]}}\imd \gr{i_{[d]}}\imd \gr{2^n_{[d]}}}
\max_{1\leqslant i_{d+1}\leqslant N}
  \abs{S_{\gr{i}}\pr{g  }  }}_{N\geq 1}$ 
  is a submartingale with respect to the filtration $\pr{\mathcal G_N}_{N\geq 1}$.
 \end{Lemma}

 \begin{proof}
 For any $E\subset [d+1]$, 
\begin{align}
 g=\sum_{J\subset E}\pr{-1}^{\abs{J}+\abs{E}} \pr{
    \E{   f\mid \f_{\infty \gr{1_J}}}  -
 \E{   f\mid \f_{\infty \gr{1_J}-s\pr{E}\gr{e_{d+1}} }  }}
\end{align}
and since the summand vanishes if $d+1$ belongs to $J$, we actually have 
\begin{equation}\label{eq:expression_de_g_simplifiee}
 g=\sum_{J\subset E\setminus \ens{d+1}}\pr{-1}^{\abs{J}+\abs{E}} \pr{
    \E{   f\mid \f_{\infty \gr{1_J}}}  -
 \E{   f\mid \f_{\infty \gr{1_J}-s\pr{E}\gr{e_{d+1}} }  }}.
\end{equation}
Consequently, $\max_{\gr{1_{[d]}}\imd \gr{i_{[d]}}\imd \gr{2^n_{[d]}}}
\max_{1\leqslant i_{d+1}\leqslant N}
  \abs{S_{\gr{i}}\pr{g  }  }$ is $\mathcal G_N$-measurable and 
\begin{equation}
\E{\max_{\gr{1_{[d]}}\imd \gr{i_{[d]}}\imd \gr{2^n_{[d]}}}
\max_{1\leqslant i_{d+1}\leqslant N}
  \abs{S_{\gr{i}}\pr{g  }  } \mid \mathcal G_{N-1}} \geq 
  \max_{\gr{1_{[d]}}\imd \gr{i_{[d]}}\imd \gr{2^n_{[d]}}}
\max_{1\leqslant i_{d+1}\leqslant N-1}
  \abs{S_{\gr{i}}\pr{g  }  },
\end{equation}
which ends the proof of Lemma~\ref{lem:submartingale_property}.

 \end{proof}
 
 \begin{Lemma}\label{lem:inclusion_Hilbert_spaces}
  The function $\sum_{k=0}^{2^{m+1}-1}U_{d+1}^kg$ belongs to 
  $\til{\h_{d, E\setminus \ens{d+1}}}$, where the latter 
  space is defined like $\h_{d,E\setminus \ens{d+1}}$, but 
  the $\sigma$-algebra $\f_{\gr{0}}$ is replaced by 
  $\f_{\infty\gr{1_{\ens{d+1}}}}$.
 \end{Lemma}
 
 \begin{proof}
 It suffices to prove that for any non-negative integer $k$, the 
 function $U_{d+1}^kg$ belongs to 
  $\til{\h_{d, E\setminus \ens{d+1}}}$. In view of 
  \eqref{eq:expression_de_g_simplifiee} and complete commutativity of the 
  filtration $\pr{T^{-\gr{i}}\f_{\gr{0}}}_{\gr{i}\in \Z^d}$, for 
  any $I\subsetneq E\setminus \ens{d+1}$, the following 
  equality holds:
  \begin{equation*}
   \E{U_{d+1}^kg\mid \f_{\infty\gr{1_{I\cup\ens{d+1}}}}}
   =U_{d+1}^k\sum_{J\subset E\setminus \ens{d+1}}\pr{-1}^{\abs{J}+\abs{E}} \pr{
    \E{   f\mid \f_{\infty \gr{1_{J\cap I}}}}  -
 \E{   f\mid \f_{\infty \gr{1_{J\cap I}}-s\pr{E}\gr{e_{d+1}} }  }}.
  \end{equation*}
 Suppose that $d+1\in E$. In this case, $s\pr{E}=1$ hence 
 $\f_{\infty \gr{1_{J\cap I}}-s\pr{E}\gr{e_{d+1}} } $ 
 and $\f_{\infty \gr{1_{J\cap I}}}$ are  
 contained in $\f_{\infty \gr{1_{E}}   }$ hence 
 $\E{U_{d+1}^kg\mid \f_{\infty\gr{1_{I\cup\ens{d+1}}}}}=0$. 
 If $d+1\notin E$, then $I\cup\ens{d+1}\subsetneq E$ hence 
 we also have $\E{U_{d+1}^kg\mid \f_{\infty\gr{1_{I\cup\ens{d+1}}}}}=0$.
 
 That $U_{d+1}^kg$ is $\f_{\infty\gr{1_{E\cup\ens{d+1}}}}$-measurable 
 when $E\neq [d+1]$ follows from \eqref{eq:expression_de_g_simplifiee}. 
 This ends the proof of Lemma~\ref{lem:inclusion_Hilbert_spaces}.
 \end{proof}
 By Lemma~\ref{lem:submartingale_property} and 
  Doob's inequality, we infer that 
  \begin{equation}
   (I)\leq 2\sum_{E\subset[d+1]}\norm{ 
   \max_{\gr{1_{[d]}}\imd \gr{i_{[d]}}\imd \gr{2^n_{[d]}}}
  \abs{S_{\gr{i_{[d]}} ,2^{m+1}}  \pr{g  }  }
   }
  \end{equation}
  We use item~\ref{it:dim_d} in the following setting: the $\sigma$-algebra 
  $\f_{\gr{0}}$ is replaced by $\til{\f_{\gr{0}}}=
  \sigma\pr{\bigcup_{l\in \Z}T_{d+1}^l\f_{\gr{0}}}$, the function 
  $f$ is replaced by $\widetilde{f}:=\sum_{j=0}^{2^{m+1}-1}U_{d+1}^jg$ (which
  belongs to  
  $\til{\h_{d, E\setminus \ens{d+1}}}$ by 
  Lemma~\ref{lem:inclusion_Hilbert_spaces}),
  and $\widetilde{\gr{n}}=(n_1,\dots,n_d)$:
  \begin{multline}\label{eq:control_of_I}
   (I)\leq 2\abs{\gr{2^{\til{\gr{n}}}}}^{1/2}K\pr{d}
   \norm{\sum_{j=0}^{2^{m+1}-1}U_{d+1}^jg}\\ +  
   \abs{\gr{2^{\til{\gr{n}}}}}^{1/2}
  \sum_{\emptyset\subsetneq J\subset [d]} 
  C\pr{d,J}\sum_{\gr{0_J}\imd \gr{i}\imd \gr{n_J}}
  \abs{\gr{2^i}}^{-1/2} \norm{ \sum_{\gr{1_J} \imd \gr{j}\imd \gr{\pr{2^i}_J}    } 
  \til{P_{d,E\setminus \ens{d+1}}}^{\gr{j}}\sum_{j=0}^{2^{m+1}-1}U_{d+1}^j
  g}.
  \end{multline}
  The sequence 
  $\pr{\sum_{j=0}^{N-1}U_{d+1}^j
   g}_{N\geq 1}$ is a martingale hence, by item~\ref{itm:contraction} of 
   Proposition~\ref{prop:Pq},
  \begin{equation}\label{eq:control_of_I_1_first_term}
  \norm{\sum_{j=0}^{2^{m+1}-1}U_{d+1}^j
  g}=2^{\frac{m+1}2} 
  \norm{g   } \leq  2 \cdot 2^{\frac{m+1}2}\norm{f}.
  \end{equation}
  We now bound the second term of \eqref{eq:control_of_I}. Let $J$ be a 
  non-empty subset of $[d]$, and let $\gr{i}$ and $\gr{j}$ be two elements of 
  $\N^d$ such that $\gr{0_J}\imd \gr{i}\imd \gr{n_J}$ and 
  and $\gr{1_J} \imd \gr{j}\imd \gr{\pr{2^i}_J}$. Since $\f_{\infty \gr{1_{F\cup\ens{d+1}}} 
  }$ is $T_{d+1}$-invariant,
  \begin{equation}
   \til{P_{d,E\setminus \ens{d+1}}}^{\gr{j}}\sum_{k=0}^{2^{m+1}-1}U_{d+1}^k
  g=\sum_{k=0}^{2^{m+1}-1}U_{d+1}^k\til{P_{d,E\setminus \ens{d+1}}}^{\gr{j}}
  g
  \end{equation}
Since the sequence 
$
\pr{\sum_{k=0}^{2^{m+1}-1}U_{d+1}^k\til{P_{d,E\setminus \ens{d+1}}}^{\gr{j}}
  g }_{N\geq 1}$
is a martingale, we derive that 
\begin{multline}\label{eq:control_of_I_1}
 \norm{ \sum_{\gr{1_J} \imd \gr{j}\imd \gr{\pr{2^i}_J}    } 
  \til{P_{d,E\setminus \ens{d+1}}}^{\gr{j}}\sum_{k=0}^{2^{m+1}-1}U_{d+1}^k
  g}\\
  =
  2^{\frac{m+1}2} 
  \norm{\sum_{F\subset  E\setminus \ens{d+1}    }\pr{-1}^{\abs{F}}
  \E{\sum_{\gr{1_J} \imd \gr{j}\imd \gr{\pr{2^i}_J}    }U^{\gr{j}\cdot \eps\pr{E}      }
  g\mid \f_{\infty \gr{1_{F\cup\ens{d+1}}}}}}.
\end{multline}

Assume that $d+1\notin E$. In this case, 
\begin{equation}
 g=f-U_{d+1}P_E^{\gr{e_{d+1}}}f=
 f-\sum_{I\subset E}\pr{-1}^{\abs{E}+\abs{I}} 
 \E{f\mid \f_{\infty\gr{1_I}-\gr{e_{d+1}}}},
\end{equation}
and since $f$ belongs to $\h_{d+1,E}$, we derive that 
$\E{f\mid \f_{\infty\gr{1_I}-\gr{e_{d+1}}}}=0$ if $I\subsetneq E$
hence 
\begin{equation}
 g=f-\E{f\mid \f_{\infty\gr{1_E}-\gr{e_{d+1}}}}.
\end{equation}
Consequently, using the $\f_{\infty\gr{1_E}}$-measurability of $f$, we derive that 
\begin{multline}
 \E{\sum_{\gr{1_J} \imd \gr{j}\imd \gr{\pr{2^i}_J}    }U^{\gr{j}\cdot \eps\pr{E}      }
  g\mid \f_{\infty \gr{1_{F\cup\ens{d+1}}}}}
  \\=\E{\sum_{\gr{1_J} \imd \gr{j}\imd \gr{\pr{2^i}_J}    }U^{\gr{j}\cdot \eps\pr{E}      }
  f\mid \f_{\infty \gr{1_{F }}}}
  -\E{\sum_{\gr{1_J} \imd \gr{j}\imd \gr{\pr{2^i}_J}    }U^{\gr{j}\cdot \eps\pr{E}      }
  f\mid \f_{\infty \gr{1_{F}}-\gr{e_{d+1}}}}
\end{multline}
and we get 
\begin{equation*}
  \norm{\sum_{F\subset  E\setminus \ens{d+1}    }\pr{-1}^{\abs{F}}
  \E{\sum_{\gr{1_J} \imd \gr{j}\imd \gr{\pr{2^i}_J}    }U^{\gr{j}\cdot \eps\pr{E}      }
  g\mid \f_{\infty \gr{1_{F\cup\ens{d+1}}}}}} \leq 
   2\norm{\sum_{\gr{1_J} \imd \gr{j}\imd \gr{\pr{2^i}_J}    } P_{d+1,E}^{\gr{j}} f }.
\end{equation*}
Assume now that $d+1$ belongs to $E$. In this case, 
\begin{equation}
 g=\E{f\mid \f_{\infty\gr{1_{E\setminus \ens{d+1}}}+\gr{e_{d+1}}   }}
\end{equation}
and we get 
\begin{equation}
 \sum_{F\subset  E\setminus \ens{d+1}    }\pr{-1}^{\abs{F}}
  \E{\sum_{\gr{1_J} \imd \gr{j}\imd \gr{\pr{2^i}_J}    }U^{\gr{j}\cdot \eps\pr{E}      }
  g\mid \f_{\infty \gr{1_{F\cup\ens{d+1}}}}} = \\
  \E{h
  \mid \f_{\infty\gr{1_{[d]}}}}
\end{equation}
where 

\begin{multline}
 h=\sum_{F\subset  E\setminus \ens{d+1}    }\pr{-1}^{\abs{F}}
  \E{\sum_{\gr{1_J} \imd \gr{j}\imd \gr{\pr{2^i}_J}    }U^{\gr{j}\cdot \eps\pr{E}      }
  f\mid \f_{\infty \gr{1_{F\cup\ens{d+1}}}}} \\
  -\sum_{F\subset  E\setminus \ens{d+1}    }\pr{-1}^{\abs{F}}
  \E{\sum_{\gr{1_J} \imd \gr{j}\imd \gr{\pr{2^i}_J}    }U^{\gr{j}\cdot \eps\pr{E}      }
  f\mid \f_{\infty \gr{1_{F }}}} 
\end{multline}
hence, in both cases,  
\begin{multline}\label{eq:control_of_I_2}
  \norm{\sum_{F\subset  E\setminus \ens{d+1}    }\pr{-1}^{\abs{F}}
  \E{\sum_{\gr{1_J} \imd \gr{j}\imd \gr{\pr{2^i}_J}    }U^{\gr{j}\cdot \eps\pr{E}      }
  g\mid \f_{\infty \gr{1_{F\cup\ens{d+1}}}}}} \\
  \leq 
   2\norm{\sum_{\gr{1_J} \imd \gr{j}\imd \gr{\pr{2^i}_J}    } P_{d+1,E}^{\gr{j}} f }.
\end{multline}
The combination of \eqref{eq:control_of_I}, \eqref{eq:control_of_I_1_first_term}, 
\eqref{eq:control_of_I_1} and \eqref{eq:control_of_I_2} yields 
\begin{multline}\label{eq:estimate_of_I_final}
   (I)\leq 4\abs{\gr{2^{\gr{n}}}}^{1/2}K\pr{d}
   \norm{f}\\ +4  
   \abs{\gr{2^{\gr{n}}}}^{1/2}
  \sum_{\emptyset\subsetneq J\subset [d]} 
  C\pr{d,J}\sum_{\gr{0_J}\imd \gr{i}\imd \gr{n_J}}
  \abs{\gr{2^i}}^{-1/2} 
  \norm{\sum_{\gr{1_J} \imd \gr{j}\imd \gr{\pr{2^i}_J} 
  } P_{d+1,E}^{\gr{j}} f }.
\end{multline}

Let us estimate the impact of $(II)$. Using inequality $\norm{\max_{i\in I}\abs{Y_i}}
\leq \sqrt{\abs I}\max_{i\in I}\norm{Y_i}$, we infer that 
\begin{equation*}
 (II)\leq 2^{m/2}
 \norm{\max_{\gr{1}\imd \gr{i}\imd \gr{2^{n-(m+1)e_{d+1}}}}
  \abs{S_{\gr{i}}\pr{U_{d+1}^{s\pr{E}} P_{d+1,E}^{\gr{e_{d+1}}}\pr{f}     }
  }}=2^{m/2}
 \norm{\max_{\gr{1}\imd \gr{i}\imd \gr{2^{n-(m+1)e_{d+1}}}}
  \abs{S_{\gr{i}}\pr{  P_{d+1,E}^{\gr{e_{d+1}}}\pr{f}     }
  }}.
\end{equation*}
By item~\ref{it:dim_d} applied to $P_{d+1,E}^{\gr{e_{d+1}}}\pr{f} $ 
instead of $f$ and $\til{\gr{n}}=(n_1,\dots,n_d)$, $\f_{\gr{0}}$ and 
$\h_{d,E}$ if 
$d+1\notin E$ (and $\f_{\infty\gr{1_{\ens{d+1}}}}$, $\til{h_{d,E\setminus 
\ens{d+1}}}$ defined in Proposition~\ref{prop:Pq}), the following inequality holds 
\begin{multline*}
 (II) \leq  2^{m/2}\abs{\gr{2^{\til{n}}}}^{1/2}K\pr{d}
 \norm{P_{d+1,E}^{\gr{e_{d+1}}}\pr{f}  }\\ + 2^{m/2}
 \abs{\gr{2^{\til{n}}}}^{1/2}
  \sum_{\emptyset\subsetneq J\subset [d]}
  C\pr{d,J}\sum_{\gr{0_J}\imd \gr{i}\imd \gr{n_J}}
  \abs{\gr{2^i}}^{-1/2} \norm{ \sum_{\gr{1_J} \imd \gr{j}\imd \gr{\pr{2^i}_J}    } 
  \til{P_{d,E\setminus \ens{d+1}}^{\gr{j}}}P_{d+1,E}^{\gr{e_{d+1}}}\pr{f}},
\end{multline*}
and using item~\ref{itm:commutativity_PU} of Proposition~\ref{prop:Pq}, it 
follows that 
\begin{multline}\label{eq:estimate_of_II_final}
 (II) \leq  2\cdot \abs{\gr{2^n}}^{1/2}K\pr{d}
 \norm{f  }\\ + \abs{\gr{2^n}}^{1/2}
  \sum_{\emptyset\subsetneq J\subset [d]}
  C\pr{d,J}\sum_{\gr{0_J}\imd \gr{i}\imd \gr{n_J}}
  \abs{\gr{2^i}}^{-1/2} \norm{ \sum_{\gr{1_J} \imd \gr{j}\imd \gr{\pr{2^i}_J}    } 
  P^{\gr{j}+\gr{e_{d+1}}} f}.
\end{multline}
We now bound $(III)$ using item~\ref{it:HDR_dim_d+1} in the following setting: we take 
$\til{T_i}=T_i$ for $i\in [d]$ and $\til{T_{d+1}}=T_{d+1}^2$, $\til{\f_{\gr{0}}}
=\f_{\gr{0}}$, $\til{f}=\pr{I+U_{d+1}}U_{d+1}^{s\pr{E}}P_{d+1,E}^{\gr{e_{d+1}}}\pr{f}$ 
and $\til{\gr{n}}=\gr{n}-\gr{e_{d+1}}$ in order to get 
\begin{multline*}
  (III)\leqslant 
  \abs{\gr{2^{\til{n}}}}^{1/2}K\pr{d+1}
  \norm{ \pr{\Id+U_{d+1}}U_{d+1}^{s\pr{E}}P_{d+1,E}^{\gr{e_{d+1}}}\pr{f}}\\ + 
  \abs{\gr{2^{\til{n}}}}^{1/2}
  \sum_{\emptyset\subsetneq J\subset [d+1]}
  C\pr{d+1,J}\sum_{\gr{0_J}\imd \gr{i}\imd \til{\gr{n_J}}}
  \abs{\gr{2^i}}^{-1/2}  \norm{ \sum_{\gr{1_J} \imd \gr{j}\imd \gr{\pr{2^i}_J}    } 
  \til{P}_{d+1,E}^{\gr{j}} U_{d+1}^{s\pr{E}}P_{d+1,E}^{\gr{e_{d+1}}}\pr{f}}.
\end{multline*}
We notice that if $d+1$ does not belong to $J$, then 
$\sum_{\gr{1_J} \imd \gr{j}\imd \gr{\pr{2^i}_J}}   \til{P}_{d+1,E}^{\gr{j}}=
\sum_{\gr{1_J} \imd \gr{j}\imd \gr{\pr{2^i}_J}}   P_{d+1,E}^{\gr{j}}$ 
and if $d+1$ belongs to $J$, then 
$\sum_{\gr{1_J} \imd \gr{j}\imd \gr{\pr{2^i}_J}}\til{P}^{\gr{j}}=
\sum_{\gr{1_J} \imd \gr{j}\imd \gr{\pr{2^i}_J}}
P_{d+1,E}^{\gr{j_{J\setminus \ens{d+1}}}+2j_{d+1}\gr{e_{d+1}}}
 $. By item~\ref{itm:commutativity_P_U_d} of Proposition~\ref{prop:Pq}, we derive that 
 \begin{multline}
  (III)\leq 2^{1/2}\abs{\gr{2^n}}^{1/2}
K\pr{d+1}
  \norm{P_{d+1,E}^{\gr{e_{d+1}}}\pr{f}}\\ + 
   2^{-1/2}\abs{\gr{2^{ n}}}^{1/2}
  \sum_{\emptyset\subsetneq J\subset [d]}
  C\pr{d+1,J}\sum_{\gr{0_J}\imd \gr{i}\imd \gr{n_J}}
  \abs{\gr{2^i}}^{-1/2}  \norm{ \sum_{\gr{1_J} \imd \gr{j}\imd \gr{\pr{2^i}_J}    } 
   P_{d+1,E}^{\gr{j}+ \gr{e_{d+1}}  }  f} +(III'),
  \end{multline}
where 
\begin{equation*}
 (III'):=2^{-1/2}\abs{\gr{2^{ n}}}^{1/2}\sum_{\emptyset\subsetneq J\subset [d]}
  C\pr{d+1,J\cup\ens{d+1}}\sum_{\gr{0_J}\imd \gr{i}\imd \gr{n_J}}\sum_{  i_{d+1}=0}
 ^{m}
  \abs{\gr{2^i}}^{-1/2}  \norm{ \sum_{\gr{1_J} \imd \gr{j}\imd \gr{\pr{2^i}_J}    }
  \sum_{j_{d+1}=1}^{2^{i_{d+1}+1}}
   P_{d+1,E}^{\gr{j}+ j_{d+1}\gr{e_{d+1}}  }  f},
\end{equation*}
hence, denoting $J':=J\cup\ens{d+1}$ for $\emptyset\subsetneq J\subset [d]$, 
and making the change of index $i'_{d+1}:=i_{d+1}+1$, 
\begin{multline}
 (III')\leq \abs{\gr{2^{ n}}}^{1/2}\sum_{\emptyset\subsetneq J\subset [d]}
  C\pr{d+1,J\cup\ens{d+1}}\sum_{\gr{0_{J'}}  \imd \gr{i}\imd 
  \gr{n_{J' }}} 
  \abs{\gr{2^i}}^{-1/2}  \norm{ \sum_{\gr{1_{J'}} \imd \gr{j}
  \imd \gr{\pr{2^i}_{J'}}    }
   P_{d+1,E}^{\gr{j}   }  f}\\
   -\abs{\gr{2^{ n}}}^{1/2}\sum_{\emptyset\subsetneq J\subset [d]}
  C\pr{d+1,J\cup\ens{d+1}}\sum_{\gr{0_J}\imd \gr{i}\imd \gr{n_J}} 
  \abs{\gr{2^i}}^{-1/2}  \norm{ \sum_{\gr{1_J} \imd \gr{j}\imd \gr{\pr{2^i}_J}    }
   P_{d+1,E}^{\gr{j}    }  f}.
\end{multline}
We derive the following inequality:
\begin{multline}\label{eq:estimate_of_III_final}
 (III)\leq 2^{1/2}\abs{\gr{2^n}}^{1/2}
K\pr{d+1}
  \norm{P_{d+1,E}^{\gr{e_{d+1}}}\pr{f}}\\ + 
   2^{-1/2}\abs{\gr{2^{ n}}}^{1/2}
  \sum_{\emptyset\subsetneq J\subset [d]}
  C\pr{d+1,J}\sum_{\gr{0_J}\imd \gr{i}\imd \gr{n_J}}
  \abs{\gr{2^i}}^{-1/2}  \norm{ \sum_{\gr{1_J} \imd \gr{j}\imd \gr{\pr{2^i}_J}    } 
   P_{d+1,E}^{\gr{j}+ \gr{e_{d+1}}  }  f} \\
   +\abs{\gr{2^{ n}}}^{1/2}\sum_{\emptyset\subsetneq J\subset [d]}
  C\pr{d+1,J\cup\ens{d+1}}\sum_{\gr{0_{J'}}  \imd \gr{i}\imd 
  \gr{n_{J' }}} 
  \abs{\gr{2^i}}^{-1/2}  \norm{ \sum_{\gr{1_{J'}} \imd \gr{j}
  \imd \gr{\pr{2^i}_{J'}}    }
   P_{d+1,E}^{\gr{j}   }  f}\\
   -\abs{\gr{2^{ n}}}^{1/2}\sum_{\emptyset\subsetneq J\subset [d]}
  C\pr{d+1,J\cup\ens{d+1}}\sum_{\gr{0_J}\imd \gr{i}\imd \gr{n_J}} 
  \abs{\gr{2^i}}^{-1/2}  \norm{ \sum_{\gr{1_J} \imd \gr{j}\imd \gr{\pr{2^i}_J}    }
   P_{d+1,E}^{\gr{j}+\gr{e_{d+1}}   }  f}.
\end{multline}
Combining \eqref{eq:decomposition_of_max}, \eqref{eq:estimate_of_I_final}, 
\eqref{eq:estimate_of_II_final} and \eqref{eq:estimate_of_III_final}, we derive that 
\begin{multline}
  \norm{\max_{\gr{1}\imd \gr{i}\imd \gr{2^n}}
  \abs{S_{\gr{i}}\pr{f}}}\leq 6\abs{\gr{2^{\gr{n}}}}^{1/2}K\pr{d}
   \norm{f}\\ +  
   \abs{\gr{2^{\gr{n}}}}^{1/2}
  \sum_{\emptyset\subsetneq J\subset [d]} 
   4C\pr{d,J} \sum_{\gr{0_J}\imd \gr{i}\imd \gr{n_J}}
  \abs{\gr{2^i}}^{-1/2} 
  \norm{\sum_{\gr{1_J} \imd \gr{j}\imd \gr{\pr{2^i}_J} 
  } P_{d+1,E}^{\gr{j}} f }\\
   + \abs{\gr{2^n}}^{1/2}
  \sum_{\emptyset\subsetneq J\subset [d]}
  \til{C\pr{J}}\sum_{\gr{0_J}\imd \gr{i}\imd \gr{n_J}}
  \abs{\gr{2^i}}^{-1/2} \norm{ \sum_{\gr{1_J} \imd \gr{j}\imd \gr{\pr{2^i}_J}    } 
  P_{d+1,E}^{\gr{j}+\gr{e_{d+1}}} f}\\
  +2^{1/2}\abs{\gr{2^n}}^{1/2}
K\pr{d+1}
  \norm{P_{d+1,E}^{\gr{e_{d+1}}}\pr{f}}\\ 
  +\abs{\gr{2^{ n}}}^{1/2}\sum_{\emptyset\subsetneq J\subset [d]}
  C\pr{d+1,J\cup\ens{d+1}}\sum_{\gr{0_{J'}}  \imd \gr{i}\imd 
  \gr{n_{J' }}} 
  \abs{\gr{2^i}}^{-1/2}  \norm{ \sum_{\gr{1_{J'}} \imd \gr{j}
  \imd \gr{\pr{2^i}_{J'}}    }
   P_{d+1,E}^{\gr{j}   }  f},
\end{multline}
where $\til{C\pr{J}}:=C\pr{d,J}+2^{-1/2}C\pr{d+1,J}- C\pr{d+1,J\cup\ens{d+1}}$
and using \eqref{eq:recurrence_Kd}, \eqref{eq:recurrence_Cd+1_J}, 
\eqref{eq:recurrence_Cd+1_J_bis} and \eqref{eq:recurrence_singleton_d+1}, we obtain 
 \eqref{eq:maximal_inequality_dim_d+1} for $n_{d+1}=m+1$. This proves the 
 first inequality in \eqref{eq:inegalite_maximale_enonce}. The second 
 one follows from a multidimensional extension of Lemma~2.7 in \cite{MR2123210}.
 \end{proof} 

\begin{proof}[Proof of Theorem~\ref{thm:approximation_sous_Maxwell_Woodroofe}]
 Using Theorem~\ref{thm:CNS_orthomartingale_approximation}, 
 we shall only check that \eqref{eq:easier_sufficient_cond} holds.
 
 Let $\emptyset \subsetneq J\subset [d]$ and $E\subset [d]$.
 An application of Proposition~\ref{prop:inegalite_maximale} reduces 
 the proof to 
\begin{equation}\label{eq:convergence_ank}
 \lim_{k\to +\infty}\frac 1{k^{\abs J}}
  \sum_{\gr{n}\smd \gr{1}} \abs{\gr{n}}^{-3/2}
   \norm{\sum_{\gr{0}\imd \gr{i}\imd\gr{n}-\gr{1}}
   P_E^{\gr{i}}\pr{\sum_{\gr{1_J}\imd \gr{j}\imd k\gr{1_J}}
 P_E^{\gr{j}}f}
  }= 0
\end{equation}

If $q$ belongs to $J$, then 
\begin{equation}
 \frac 1{k^{\abs{J}}}\norm{\sum_{\gr{0}\imd \gr{i}\imd\gr{n}-\gr{1}}
 \sum_{\gr{1_J}\imd \gr{j} \imd k\gr{1_J}   
  }
  P_E^{\gr{i}+\gr{j}}f
  }
  \leq \frac 1k\norm{\sum_{\gr{0}\imd \gr{i}\imd\gr{n}-\gr{1}}\sum_{l=1  
  }^k
  P_E^{\gr{i}+l\gr{e_q}}f
  }
\end{equation}
hence it suffices to prove that for any $q\in [d]$, 
\begin{equation}
  \lim_{k\to +\infty}\frac 1{k}
  \sum_{\gr{n}\smd \gr{0}} a_{\gr{n},k}= 0
\end{equation}
where 
\begin{equation*}
a_{\gr{n},k}:= k^{-1}\abs{\gr{n}}^{-3/2} 
   \norm{\sum_{\gr{0}\imd \gr{i}\imd\gr{n}-\gr{1}}\sum_{l=1
  }^k
  P_E^{\gr{i}+l\gr{e_q}}f
  } .
\end{equation*}
We first observe that for any fixed $\gr{n}\smd\gr{1}$, 
$a_{\gr{n},k}\leq k^{-1}
\norm{ \sum_{l=1
  }^k
  P_E^{l\gr{e_q}}\pr{f}
  } $
  and applying Lemma~2.8. of  \cite{MR2123210} to the subadditive sequence 
  $\left( \norm{ \sum_{l=1
  }^k
  P_E^{l\gr{e_q}}\pr{f}
  } \right)_{k\geq 1}$, we derive that $a_{\gr{n},k}\to 0$ as $k$ goes to infinity. 
  Moreover, 
  \begin{equation}
   \sup_{k\geqslant 1}a_{\gr{n},k} 
  \leqslant  \abs{\gr{n} }^{-3/2}
   \norm{\sum_{\gr{0}\imd \gr{i}\imd\gr{n}-\gr{1}} 
  P_E^{\gr{i} }\pr{f}},
  \end{equation}
 hence by dominated convergence, \eqref{eq:convergence_ank} holds.
 This ends the proof of 
 Theorem~\ref{thm:approximation_sous_Maxwell_Woodroofe}.
\end{proof}

\begin{proof}[Proof of Corollary~\ref{cor:linear_processes}]
The computation of $P_E^{\gr{i}}f$ gives 
\begin{equation}
 P_E^{\gr{i}}f= \sum_{\gr{j}\in \N^d,\gr{j}\cdot \gr{1_E} \smd \gr{1_E} 
  }a_{\pr{\gr{i}+\gr{j}}\cdot\eps\pr{E}}\cdot \eps_{\gr{j}}.
\end{equation}
Summing over $\gr{i}\in [\gr{0},\gr{n}]$, taking the $\mathbb L^2$-norm 
and using orthogonality of $\eps_{\gr{j}}$'s, we derive that 
$\norm{\sum_{\gr{0}\imd\gr{i}\imd\gr{n}-\gr{1}}P_E^{\gr{i}}f}
\leq \Delta_{E,\gr{n}}$ hence \eqref{eq:linear_processes_MW} implies 
\eqref{eq:MW_dim_d}. The approximating martingale satisfies the 
invariance principle since $T^{\gr{e_1}}$ is ergodic.
\end{proof}

\newcommand{\etalchar}[1]{$^{#1}$}
\def\polhk\#1{\setbox0=\hbox{\#1}{{\o}oalign{\hidewidth
  \lower1.5ex\hbox{`}\hidewidth\crcr\unhbox0}}}\def\cprime{$'$}
  \def\cprime{$'$} \def\rasp{\leavevmode\raise.45ex\hbox{$\rhook$}}
  \def\cftil#1{\ifmmode\setbox7\hbox{$\accent"5E#1$}\else
  \setbox7\hbox{\accent"5E#1}\penalty 10000\relax\fi\raise 1\ht7
  \hbox{\lower1.15ex\hbox to 1\wd7{\hss\accent"7E\hss}}\penalty 10000
  \hskip-1\wd7\penalty 10000\box7}
  \def\polhk#1{\setbox0=\hbox{#1}{\ooalign{\hidewidth
  \lower1.5ex\hbox{`}\hidewidth\crcr\unhbox0}}} \def\cprime{$'$}
\providecommand{\bysame}{\leavevmode\hbox to3em{\hrulefill}\thinspace}
\providecommand{\MR}{\relax\ifhmode\unskip\space\fi MR }
\providecommand{\MRhref}[2]{%
  \href{http://www.ams.org/mathscinet-getitem?mr=#1}{#2}
}
\providecommand{\href}[2]{#2}

\end{document}